\documentclass[journal,web]{IEEEtran}

\usepackage{cite}
\usepackage{amsmath,amssymb,amsfonts,amsthm}
\usepackage{algorithmic}
\usepackage{graphicx}
\usepackage{textcomp}

\def\BibTeX{{\rm B\kern-.05em{\sc i\kern-.025em b}\kern-.08em
    T\kern-.1667em\lower.7ex\hbox{E}\kern-.125emX}}

\usepackage{enumitem}

\usepackage{balance}
\usepackage[]{xcolor}
\usepackage{xfrac}
\usepackage{tikz}
\usepackage{pgf}
\usepackage{pgfplots}
\usetikzlibrary{positioning, calc}
\pgfplotsset{compat=1.13}
\usepgfplotslibrary{groupplots}

\usepackage[caption=false]{subfig}

\usepackage{mathtools}
\usepackage{newfloat}
\DeclarePairedDelimiter{\abs}{\lvert}{\rvert}
\DeclarePairedDelimiter{\norm}{\lVert}{\rVert}
\DeclarePairedDelimiter{\ceil}{\lceil}{\rceil}

\newcommand{\sqrtm}[1]{#1^{\sfrac{1}{2}}}
\newcommand{\trans}[1]{#1'}
\newcommand{\ctrans}[1]{#1^{*}}
\newcommand{\trace}[1]{\text{Tr}\left(#1\right)}
\newcommand{\Htwonorm}[1]{\norm{#1}_{H_2}}
\newcommand{\Hinfnorm}[1]{\norm{#1}_{H_\infty}}
\newcommand{\Fronorm}[1]{\norm{#1}_{F}}
\newcommand{\Twonorm}[1]{\norm{#1}_{2}}
\newcommand{\Linf}[0]{\mathcal{L}^{\infty}}
\newcommand{\Cpi}{\mathcal{\tilde{C}}_{2\pi}}

\newtheorem{proposition}{Proposition}
\newtheorem{assumption}{Assumption}
\newtheorem{theorem}{Theorem}
\newtheorem{lemma}{Lemma}
\newtheorem{corollary}{Corollary} 
\newtheorem{system}{System}

\newtheorem{definition}{Definition}

\begin{document}
    
\title{Bounding Computational Complexity under Cost Function Scaling in Predictive Control}

\author{Ian~McInerney, Eric~C.~Kerrigan, George~A.~Constantinides
\thanks{The support of the EPSRC Centre for Doctoral Training in High Performance Embedded and Distributed Systems  (HiPEDS, Grant Reference EP/L016796/1) is gratefully acknowledged.}%
\thanks{The authors are with the Department of Electrical \& Electronics Engineering, Imperial College London, SW7~2AZ, U.K. E.C.\ Kerrigan is also with the  Department of Aeronautics. email: i.mcinerney17@imperial.ac.uk,  e.kerrigan@imperial.ac.uk, g.constantinides@imperial.ac.uk}
\thanks{Supporting code for the analysis and figures in this paper is available at {\tt https://doi.org/10.24433/CO.3311801.v1}.}}

\maketitle

\begin{abstract}
We present a framework for upper bounding the number of iterations required by first-order optimization algorithms implementing constrained LQR controllers.
 We derive new bounds for the condition number and extremal eigenvalues of the primal and dual Hessian matrices when the cost function is scaled.
 These bounds are horizon-independent, allowing for their use with receding, variable and decreasing horizon controllers.
 We considerably relax prior assumptions on the structure of the weight matrices and assume only that the system is Schur-stable and the primal Hessian of the quadratic program (QP) is positive-definite.
 Our analysis uses the Toeplitz structure of the QP matrices to relate their spectrum to the transfer function of the system, allowing for the use of system-theoretic techniques to compute the bounds.
 Using these bounds, we can compute the effect on the computational complexity of trading off the input energy used against the state deviation.
 An example system shows a three-times increase in algorithm iterations between the two extremes, with the state 2-norm decreased by only 5\% despite a greatly increased state deviation penalty.
\end{abstract}


\section{Introduction}
\label{sec:intro}

Processors in modern Cyber-Physical Systems (CPS) are routinely being utilized for more than just control, with additional tasks such as communication, coordination, user-interface and data-collection becoming more widespread as designers adopt networked systems and the Internet of Things.
 At the same time, computational resources are being further constrained by the demand for low-power and low-cost designs.
 Guaranteeing the proper operation of the control system on resource constrained processors in environments with safety and operational constraints is important to guarantee the dependability of the CPS \cite{Johansen2017_DependableMPC}.
 
Model Predictive Control (MPC), and specifically the Constrained Linear Quadratic Regulator (CLQR), was recently  highlighted in \cite{Lucia2016_CPSperspectives} as a control algorithm aptly suited to provide these operational guarantees at a functional level.
 This means that guaranteeing the dependable operation of the CPS now requires analyzing the MPC algorithm, and guaranteeing its performance given the computational resources present.
 To this end, the authors of \cite{Johansen2017_DependableMPC} suggest that an important question to answer is: what effect do the reduced computational resources have on the system performance?
 This question views the computation and system performance as two separate factors, focusing on quantifying the performance degradation experienced due to the computational resources.

Instead, we propose that the system performance should be examined together with the computational resources by asking the question: what effect does the desired system performance have on the computational resources required?
 The performance constraints are usually given as bounds (e.g.\ ``settling-time less than 1s" or ``track this signal with less than 5\% error") rather than an exact criterion, creating a space of possible controllers that can satisfy the requirements.
 These controllers may have different computational resource demands, opening up the opportunity to trade-off the computation with the system performance.
 We show in this paper that exploring this trade-off is beneficial; with a sample system demonstrating a reduction in computational resource demand by an order of magnitude while only increasing the state 2-norm value by 5\%.
 
To quantify this trade-off, we analyze how the computational complexity (e.g.\ number of algorithm iterations) varies as the weighting matrices in the cost function of the CLQR are changed.
 We focus on first-order methods for solving the Quadratic Programming (QP) formulation of the condensed CLQR method (such as the Fast Gradient Method (FGM)~\cite{Richter2012} and Dual Gradient Projection (DGP) method~\cite{Patrinos2013_DGP}), since their computational complexity is sensitive to the problem conditioning \cite{Richter2012,Giselsson2013_optimalPrecond,Giselsson2014_precondFastDual}.
 For this analysis, we present a framework to compute horizon-independent condition numbers and extremal eigenvalues for the matrices in both the primal and dual QPs for the condensed CLQR with arbitrary weighting matrices for the states, inputs and state-input cross-term (defined as $Q$, $R$ and $S$, respectively).
 We then use these results to examine the effect that scaling the weight matrices has on the computational complexity of FGM and DGP.
 
Prior work in \cite{Richter2012} and \cite{Goodwin2005} derived horizon-independent bounds on the condition number of the primal QP with no cross-term $S$ and with the structure of the input weight matrix $R$ assumed to be a multiple of the identity matrix.
 Our relaxation of the assumptions on $R$ and $S$ allow for the bounding of problems created through controller matching (e.g.\ \cite{DiCairano2010_MPCcontrollerTuning,Hartley2013_LTIcontrollerMatching}) or through the discretization of a continuous-time problem into discrete-time \cite{Bini2014_OptimalSamplingPattern}.
 Prior bounds for the dual QP matrices in \cite{Patrinos2013_DGP} utilize the sub-multiplicative property of the matrix 2-norm to bound the maximal eigenvalue.
 We present a tight horizon-independent bound on the maximal eigenvalue in this work, reducing the conservativeness of iteration bounds for dual methods such as DGP.

Our bounding framework is based on truncated infinite-dimensional Toeplitz operators (see \cite{Gutierrez-Gutierrez2012_blockSurvey, Gray2005} for a survey of the mathematics behind these operators).
 These operators have been applied extensively in the field of robust control~\cite{Francis1987}, but they have seen limited application to MPC.
 Prior work in MPC using these operators has focused on deriving suboptimal MPC algorithms~\cite{Rojas2003_SVDcontrol,Kojima2004_ContSVD},
 equating properties of the condensed primal Hessian matrix with the system frequency response~\cite{Rojas2004_HessAsymp},
 and relating the stability of MPC to the phase-space of the system \cite{Tran2012_SISOasymp,Tran2013_MIMOasymp}.
 We use these operators to find the horizon-independent bounds using system-theoretic properties such as the $H_{\infty}$ norm.

In the remainder of this section we define the mathematical notation used throughout the rest of the paper.
 In Section~\ref{sec:prelim} we present the CLQR problem formulation, an overview of the computational complexity of FGM and DGP, and the dynamical systems used in the numerical examples throughout this paper.
 We present the spectral bounds of the QP matrices in Section~\ref{sec:specRes}.
 Section~\ref{sec:weightScaling} examines the effect of weight matrix scaling on the condition number and extremal eigenvalues of the QP matrices.
  Section~\ref{sec:compComplexity} then examines the effect of the weight matrix scaling on the computational complexity of FGM and DGP.
 We briefly discuss the extension of these results to the case when the QP is preconditioned in Section~\ref{sec:preconditioning}, before concluding in Section~\ref{sec:conclusion}.

\subsection{Notation}

$\trans{A}$ and $\ctrans{A}$ denote the transpose and conjugate-transpose of the matrix $A$ respectively.
 $A\otimes B$ represents the Kronecker product of matrix $A$ with matrix $B$.
 For a block-Toeplitz matrix $T$, $T_{n}$ represents the truncated version of $T$ after $n$ diagonals (where $n$ is a non-zero positive integer).
 Let $\lambda_1 \leq \dots \leq \lambda_k$ be the eigenvalues of a matrix in sorted order, with the set of all eigenvalues denoted by $\lambda$.
 Let $0 \leq \sigma_{k} \leq \dots \leq \sigma_{1}$ be the singular values of a matrix in sorted order, with the set of all singular values denoted by $\sigma$.
 The p-norm is denoted by $\norm{\circ}_p$, with $\Twonorm{A}$ the matrix spectral norm, and $\Fronorm{A}$ the Frobenius norm.
 The $H_{2}$ and $H_{\infty}$ norms of a dynamical system $G_{s}(\cdot)$ are $\Htwonorm{G_{s}}$ and $\Hinfnorm{G_{s}}$ respectively.
 The condition number of a matrix is defined as $\kappa(A)\coloneqq\nobreak\Twonorm{A}\Twonorm{A^{-1}}$, 
 and the condition number of a dynamical system is defined as $\kappa(G_{s})\coloneqq\nobreak\Hinfnorm{G_{s}}\Hinfnorm{G_{s}^{-1}}$. 
 
$\Linf$ is the space of matrix-valued essentially bounded functions (e.g.\ matrix-valued functions that are measurable and have a finite Frobenious norm almost-everywhere on their domain, see \cite[\S 2]{Miranda2000}).
 $\Cpi$ is the space of continuous $2\pi$-periodic functions inside $\Linf$.
\begin{definition}
    Let $\mathbb{T} \coloneqq \{z \in \mathbb{C} : \abs{z} = 1 \}$ be the complex unit circle and $P_{T}(\cdot)$ a function that maps $\mathbb{T} \to \mathbb{C}^{m \times n}$.
    With $k = \min\{m,n\}$, the extreme singular values of $P_{T}(\cdot)$ are:
    \begin{align*}
        \sigma_{min}(P_{T})  \coloneqq \underset{z \in \mathbb{T}}{\sup}~\sigma_k(P_{T}(z)), \quad
        \sigma_{max}(P_{T})  \coloneqq \underset{z \in \mathbb{T}}{\sup}~\sigma_1(P_{T}(z)).
    \end{align*}
    If $m = n$, then the extreme eigenvalues of $P_{T}(\cdot)$ are:
    \begin{align*}
        \lambda_{min}(P_{T}) \coloneqq \underset{z \in \mathbb{T}}{\inf}~\lambda_1(P_{T}(z)), \quad
        \lambda_{max}(P_{T}) \coloneqq \underset{z \in \mathbb{T}}{\sup}~\lambda_n(P_{T}(z)).
    \end{align*}
\end{definition}
\begin{definition}
    Let $T_{n}$ be the $n \times n$ truncation of an infinite matrix $T$. We define the extrema of the spectrum of $T$ as:
    \begin{align*}
    \lambda_{min}(T) &\coloneqq \lim_{n \to \infty}~\lambda_1(T_{n}), \quad 
    \lambda_{max}(T) \coloneqq \lim_{n \to \infty}~\lambda_n(T_{n}), \\
    \sigma_{min}(T)  &\coloneqq \lim_{n \to \infty}~\sigma_n(T_{n}), \quad
    \sigma_{max}(T)  \coloneqq \lim_{n \to \infty}~\sigma_1(T_{n}).
    \end{align*}
\end{definition}

\section{MPC Preliminaries}
\label{sec:prelim}

In this section we present the constrained LQR formulation of the MPC problem.
 We introduce both the primal and dual quadratic programming formulations, and first-order optimization algorithms to solve them.

\subsection{CLQR Formulation}
\label{sec:prelim:formulation}

The CLQR formulation of MPC can be written as the following constrained quadratic programming problem
\begin{subequations}
    \label{eq:mpc:linMPC}
    \begin{align}
    \underset{u,x}{\text{min}}\   & \frac{1}{2} \trans{x_N} P x_N + \frac{1}{2} \sum_{k=0}^{N-1} 
    \trans{ \begin{bmatrix} x_k\\ u_k \end{bmatrix} }
    \begin{bmatrix}
    Q & S\\
    \trans{S} & R
    \end{bmatrix}
    \begin{bmatrix} x_k\\ u_k \end{bmatrix}
    \label{eq:mpc:lin:cost}\\
    \text{s.t.\ }  &
    \begin{aligned}[t]
    x_{k+1} &= A x_k + B u_k,\ k=0, \dots N-1 \label{eq:mpc:lin:dyn} \\
    x_{0} &= \hat{x}_{0}
    \end{aligned}\\
    & E u_k \leq c_u,\ k=0, \dots N-1 \label{eq:mpc:lin:conU}\\
    & D x_k \leq c_x,\ k=1, \dots N \label{eq:mpc:lin:conX}
    \end{align}
\end{subequations}
 where $N$ is the horizon length, $x_k \in \mathbb{R}^{n}$ are the states, and $u_k \in \mathbb{R}^{m}$ are the inputs at time sample $k$.
 $A \in \mathbb{R}^{n \times n}$ and $B \in \mathbb{R}^{n \times m}$ are the state-space matrices describing the discrete-time system $G_{s}$, and $\hat{x}_0 \in \mathbb{R}^{n}$ is the current measured system state.
 $D \in \mathbb{R}^{j \times n}$ and $E \in \mathbb{R}^{l \times m}$ are the stage constraint matrices for the states and inputs respectively, and the vectors $c_x \in \mathbb{R}^{j}$ and $c_u \in \mathbb{R}^{l}$ are the upper bounds for the stage constraints.
 The matrices $Q = \trans{Q} \in \mathbb{R}^{n \times n}, R = \trans{R} \in \mathbb{R}^{m \times m}, P=\trans{P} \in \mathbb{R}^{n \times n}$ are the weighting matrices for the system states, inputs, and final states respectively, and the matrix $S \in \mathbb{R}^{n \times m}$ weights the inputs against the states.
 The weighting matrices are chosen such that
 $
 \begin{bmatrix}
 Q & S\\
 \trans{S} & R
 \end{bmatrix}
 $ is positive definite.

This problem can be condensed by removing the state variables from~\eqref{eq:mpc:linMPC} to leave only the control inputs in the vector
$ u {\coloneqq} \trans{\begin{bmatrix}
\trans{u_0},~\trans{u_1},~\dots,~\trans{u_{N-1}}
\end{bmatrix}}$.
 The optimization problem is then the inequality-constrained problem
 \begin{subequations}
    \label{eq:mpc:condMPC}
    \begin{align}
    \underset{u}{\text{min}}\   & \frac{1}{2} \trans{u} H_c u + \trans{\hat{x}}_0 \trans{J} u\label{eq:mpc:cond:cost}\\
    \text{s.t.\ }  & G u \leq F \hat{x_0} + g \label{eq:mpc:cond:con}
    \end{align}
 \end{subequations}
 with 
 $ H_{c} \coloneqq \trans{\Gamma} \bar{Q} \Gamma + \trans{\bar{S}} \Gamma + \trans{\Gamma} \bar{S} + \bar{R} $,
 $ J \coloneqq \trans{\Gamma} \bar{Q} \Phi $, and the remaining matrices given in Appendix~\ref{app:mpcMatrices}.
 An alternative method of solving the CLQR problem~\eqref{eq:mpc:linMPC} is to transform~\eqref{eq:mpc:condMPC} into its dual problem 
 \begin{subequations}
    \label{eq:mpc:dual}
    \begin{align}
        \underset{y}{\text{min}}\   & \frac{1}{2} y^T H_d y + (J_d \hat{x}_0 + g)^{T} y\label{eq:mpc:dual:cost}\\
        \text{s.t.\ }  & y \geq 0 \label{eq:mpc:dual:con}
    \end{align}
 \end{subequations}
 where $y$ are the dual variables for the inequality constraints~\eqref{eq:mpc:cond:con},
 $H_d \coloneqq G H_c^{-1} \trans{G}$ and 
 $J_d \coloneqq G H_c^{-1} J + F $.
 Since there is strong duality between~\eqref{eq:mpc:condMPC} and~\eqref{eq:mpc:dual}, the control sequence can be recovered from a dual-optimal solution $y^{*}$ through
 $ u^{*} = - H_c^{-1} (\trans{G} y^{*} + \trans{J} x_0) $.
 
\subsection{Solution Algorithms}
\label{sec:prelim:algorithms}

In this work we focus on two first-order methods commonly used for embedded implementations: the Fast Gradient Method (FGM) and the Dual Gradient Projection (DGP) method.

\subsubsection{Fast Gradient Method}
\label{sec:prelim:algorithms:fgm}

The Fast Gradient Method was originally proposed by Nesterov to solve constrained convex programs, and was adapted in \cite{Richter2012} to solve the condensed QP~\eqref{eq:mpc:condMPC} with only input constraints.
 The FGM will minimize the cost~\eqref{eq:mpc:cond:cost} while using a projection operation to satisfy the inequality constraints~\eqref{eq:mpc:cond:con}.
 It is popular in embedded implementations when the constraint set~\eqref{eq:mpc:cond:con} is upper/lower bounds due to the existence of a simple projection operator, an Upper Iteration Bound (UIB), and sizing rules for fixed-point data-types.

The UIB for FGM can be computed from
 \begin{equation}
    \label{eq:fgm:uib}
    \text{UIB} \coloneqq \max\left\{ 0, \min \left\{ a, b \right\} \right\},
 \end{equation}
 \begin{equation*}
    a \coloneqq \ceil*{ \frac{\ln\epsilon - \ln\Delta}{\ln \left(1 - \sqrt{\frac{1}{\kappa}} \right)} }, \quad b \coloneqq \ceil*{ 2\sqrt{\frac{\Delta}{\epsilon}} - 2, },
 \end{equation*}
 where $\epsilon$ is the desired tolerance for the primal solution, $\kappa$ is the condition number of the Hessian $H_{c}$, and $\Delta$ is a constant determined by the constraint set \cite{Richter2012}.
 
The value of $\Delta$ can be found by either solving one of the optimization problems given in \cite{Richter2012} or utilizing an upper-bound, while the value for $\epsilon$ should be chosen to ensure stability properties of the system.
 A horizon-independent bound on $\epsilon$ and $\Delta$ for warm-started FGM was calculated in \cite{Richter2012} as
 \begin{equation}
    \label{eq:fgm:constBounds}
     \epsilon \leq \frac{\mu}{2} \frac{\delta_{max}^2}{\norm{B}^2}, \qquad
     \Delta \leq \lim\limits_{N \to \infty} \kappa \epsilon,
 \end{equation}
 where $\mu$ is the smallest eigenvalue of the Hessian $H_{c}$, and $\norm{x_{+} - x_{+}^{*}} \leq \delta_{max}$, where $x_{+}$ is the next state and $x_{+}^{*}$ is the next state under the optimal input.
 
\subsubsection{Dual Gradient Projection}
\label{sec:prelim:algorithms:dgp}
 
An alternative algorithm to solve~\eqref{eq:mpc:linMPC} is the Dual Gradient Projection (DGP) algorithm described in \cite{Patrinos2013_DGP}.
 DGP is a non-accelerated gradient method that operates on the dual problem~\eqref{eq:mpc:dual}.
 There are two main steps in the algorithm: the gradient computation and then a projection onto the non-negative orthant.
 The projection operation is onto the non-negative orthant regardless of the constraint set~\eqref{eq:mpc:cond:con}, making DGP suitable for embedded applications with complex constraint sets.
 
Theoretical results in \cite{Patrinos2013_DGP} presented the following UIB for the DGP algorithm
 \begin{equation}
    \label{eq:dgp:uib}
    \text{UIB} \coloneqq \frac{LD^2 \alpha^2}{2(\epsilon_g - 2 D \epsilon_{\xi})\alpha - 2(\epsilon_g + L_V \epsilon_z^2)} - 1,
 \end{equation}
 where $L \coloneqq \lambda_{max}(H_{d})$ and $L_V \coloneqq \lambda_{max}(H_c)$.
 $\epsilon_g$ is the largest permissible constraint satisfaction error, $\epsilon_z$ is the desired tolerance of the dual solution, and $\epsilon_{\xi}$ is the error in the dual gradient computation.
 $D$ is the Upper Dual Bound (UDB), which is defined as $D \coloneqq \norm{d}$ where $d_{i} \coloneqq \max\{y^*_i, 1\}$ and $y^*$ is the optimal dual vector.
 $D$ can be estimated by solving one of several optimization problems given in \cite{Patrinos2014_GPAD} or \cite{Necoara2014}, and $\alpha$ can be calculated using the formula in \cite{Patrinos2013_DGP}.


\subsection{Systems for the Numerical Examples}

Throughout this paper, the theoretical results are illustrated by numerical examples using the following two dynamical systems.

\begin{system}
    \label{sys:jonesMorari}
    The discrete-time system with four-states and two-inputs given in \cite{Jones2008} with state equation and cost matrices
    \begin{align*}
    \label{eq:app:cond:system}
    x^{+} = \begin{bmatrix}
    0.7 & -0.1 & 0.0 & 0.0 \\
    0.2 & -0.5 & 0.1 & 0.0 \\
    0.0 &  0.1 & 0.1 & 0.0 \\
    0.5 &  0.0 & 0.5 & 0.5
    \end{bmatrix}
    x + 
    \begin{bmatrix}
    0.0 & 0.1 \\
    0.1 & 1.0 \\
    0.1 & 0.0 \\
    0.0 & 0.0
    \end{bmatrix}
    u, \\
    Q = \text{diag}(10, 20, 30, 40), \quad
    R = \text{diag}(10, 20).
    \end{align*}
    The inputs and states of the system are constrained to be $\abs{u_{i}} \leq 0.5$ and $\abs{x_{i}} \leq 0.5$ respectively.
\end{system}

\begin{system}
    \label{sys:msd}
    \begin{figure}
        \centering
        \begin{tikzpicture}

\usetikzlibrary{calc,patterns,decorations.pathmorphing,decorations.markings}

\tikzstyle{spring}=[thick,decorate,decoration={zigzag,pre length=0.3cm,post length=0.3cm,segment length=6}]
\tikzstyle{damper}=[thick,decoration={markings,  
  mark connection node=dmp,
  mark=at position 0.5 with 
  {
    \node (dmp) [thick,inner sep=0pt,transform shape,rotate=-90,minimum width=15pt,minimum height=3pt,draw=none] {};
    \draw [thick] ($(dmp.north east)+(2pt,0)$) -- (dmp.south east) -- (dmp.south west) -- ($(dmp.north west)+(2pt,0)$);
    \draw [thick] ($(dmp.north)+(0,-5pt)$) -- ($(dmp.north)+(0,5pt)$);
  }
}, decorate]
\tikzstyle{ground}=[fill,pattern=north east lines,draw=none,minimum width=0.75cm,minimum height=0.3cm]

\node [ground, name=ground1, rotate=-90, minimum width=2cm] at (0,0) {};
\draw (ground1.north west) -- (ground1.north east);

\node [draw, name=m1, minimum height=2cm, minimum width=1cm] at (5em,0em) {\footnotesize $m_1$};
\node [draw, name=m2, minimum height=2cm, minimum width=1cm] at (11em,0em) {\footnotesize $m_2$};
\node [draw, name=m3, minimum height=2cm, minimum width=1cm] at (19em,0em) {\footnotesize $m_{10}$};

\draw [damper] ($(ground1.north west)+(0cm,-0.66cm)$) -- node [midway, above=0.75em]  {\footnotesize $b_1$} ($(m1.north west)+((0cm,-0.66cm)$);
\draw [spring] ($(ground1.north east)+(0cm,0.66cm)$) -- node [midway, below=0.5em]  {\footnotesize $k_1$} ($(m1.south west)+((0cm,0.66cm)$);

\draw [damper] ($(m1.north east)+(0cm,-0.66cm)$) -- node [midway, above=0.75em]  {\footnotesize $b_2$} ($(m2.north west)+((0cm,-0.66cm)$);
\draw [spring] ($(m1.south east)+(0cm,0.66cm)$) -- node [midway, below=0.5em]  {\footnotesize $k_2$} ($(m2.south west)+((0cm,0.66cm)$);

\draw [damper] ($(m3.north east)+(-6em,-0.66cm)$) -- node [midway, above=0.75em]  {\footnotesize $b_{10}$} ($(m3.north west)+((0cm,-0.66cm)$);
\draw [spring] ($(m3.south east)+(-6em,0.66cm)$) -- node [midway, below=0.5em]  {\footnotesize $k_{10}$} ($(m3.south west)+((0cm,0.66cm)$);

\node [draw, fill, circle, inner sep=0pt, minimum size=0.25em] at (13em,0em) {};
\node [draw, fill, circle, inner sep=0pt, minimum size=0.25em] at (13.5em,0em) {};
\node [draw, fill, circle, inner sep=0pt, minimum size=0.25em] at (14em,0em) {};

\draw [->] ($(m1.center)+(-0.5em,1em)$) -| ($(m1.center)+(-0.5em,1.5em)$) -- ($(m1.center)+(0em,1.5em)$);
\node [anchor=west] at ($(m1.center)+(-0.25em,1.5em)$) {\footnotesize $F_1$};

\draw [->] ($(m2.center)+(-0.5em,1em)$) -| ($(m2.center)+(-0.5em,1.5em)$) -- ($(m2.center)+(0em,1.5em)$);
\node [anchor=west] at ($(m2.center)+(-0.25em,1.5em)$) {\footnotesize $F_2$};

\draw [->] ($(m3.center)+(-0.5em,1em)$) -| ($(m3.center)+(-0.5em,1.5em)$) -- ($(m3.center)+(0em,1.5em)$);
\node [anchor=west] at ($(m3.center)+(-0.25em,1.5em)$) {\footnotesize $F_{10}$};
\end{tikzpicture}
        \caption{Configuration of the mass-spring-damper in System~\ref{sys:msd}.}
        \label{fig:sys:msd}
    \end{figure}
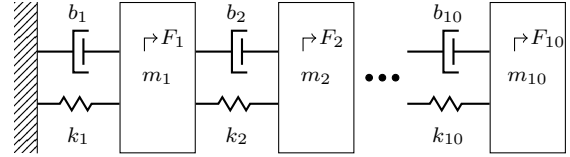
    The mass-spring-damper system from \cite{Khusainov2017_codesign} where 10 masses are coupled together by a spring and a damper in parallel, with a force input on each mass as shown in Figure~\ref{fig:sys:msd}. The continuous-time system is discretized using a zero-order hold with a sampling time of $T_{s} = 0.1$. The continuous-time cost matrices are
    \begin{equation*}
    Q_{c} = \begin{bmatrix}
    10 &  0\\
    0  & 20
    \end{bmatrix} \otimes I_{10}, \quad
    R_{c} = \text{diag}(100, 200, ...~900 , 1000),
    \end{equation*}
    which are then discretized as discussed in 
    \cite{Bini2014_OptimalSamplingPattern}
    in order to make the cost of the discrete-time problem equivalent to the cost of the continuous-time problem.
    This leads to dense $Q$ and $R$ matrices along with a cross-term matrix $S$ in the discrete-time problem.
    The inputs and states of the system are constrained to be $\abs{u_{i}} \leq 1$ and $\abs{x_{i}} \leq 0.2$ respectively.
\end{system}

\subsection{Assumptions}

In this work we make the following assumption on the predicted system in the CLQR problem.
\begin{assumption}
    The predicted system $G_{s}(\cdot)$ defined by the dynamics in \eqref{eq:mpc:lin:dyn} is Schur-stable. This means that all eigenvalues of the state transition matrix $A$ lie strictly inside the unit circle.
\end{assumption}

The assumption of Schur-stability is required to guarantee the convergence of a matrix series in Section~\ref{sec:specRes:condPred}, and therefore lies behind every result in this paper.
 This assumption is not restrictive though, since a system that is not Schur-stable can always be pre-stabilized by a separate controller that can guarantee Schur-stablity \cite{Rossiter1998_prestabilization}.

\section{Spectral Properties}
\label{sec:specRes}

In this section, we present spectral properties (e.g.\ condition numbers and eigen/singular value estimates) for the prediction matrix $\Gamma$ and the Hessian matrices $H_c$ and $H_d$.

\subsection{Condensed Prediction Matrix}
\label{sec:specRes:condPred}

The results contained in this subsection were previously reported in \cite[\S 11]{Goodwin2005} and \cite{Rojas2004_HessAsymp}, but we present an alternative derivation using Toeplitz operator theory instead of the Fourier transform to allow the results to be used in future sections.

To derive the singular value properties of the prediction matrix for the condensed form, $\Gamma$, we start by noting that its diagonals are constant blocks, making it a truncated block Toeplitz matrix.
 Many properties of a Toeplitz matrix with blocks of size $m \times n$ are closely linked to properties of a matrix-valued function mapping $\mathbb{T} \to \mathbb{C}^{m \times n}$, which is called its matrix symbol.
 The diagonal blocks of the matrix give the spectral coefficients of the matrix symbol, so the symbol can be represented as a Fourier series with the coefficients given by the matrix blocks.
 For $\Gamma$, this Fourier series converges if the discrete-time system is Schur-stable, giving the symbol in Lemma~\ref{lem:gammaSymbol}.
\begin{lemma}
    \label{lem:gammaSymbol}
    For a Schur-stable system $G_s$, the prediction matrix $\Gamma$ has the matrix symbol $P_{\Gamma} \in \Cpi$ with
    \begin{equation*}
        \label{eq:dense:pred:genFunc}
        P_{\Gamma}(z) \coloneqq z(zI - A)^{-1}B = z G_{s}(z)
        \qquad
        \forall z \in \mathbb{T}
    \end{equation*}
    where $G_s(\cdot)$ is the transfer function matrix for the system $G_s$.
\end{lemma}
\begin{proof}
    The diagonals of $\Gamma$  are composed of constant blocks of the form 
    \begin{equation}
        \label{eq:predMat:specCoeff}
        A^{i}B
    \end{equation}
    where $i$ is the diagonal number ($0$ is the main diagonal).
    This constant-diagonal structure means that $\Gamma$ is a block-Toeplitz matrix.
    To create the matrix symbol for $\Gamma$, form the trigonometric polynomial of the Fourier series using the diagonal blocks \eqref{eq:predMat:specCoeff} as the coefficients
    \begin{equation}
        \label{eq:predMat:trigPoly}
        P_{\Gamma}(z) = \sum_{i=0}^{\infty} A^iBz^{-i}
    \end{equation}
    where $z \in \mathbb{T}$.
    Since $B$ is a constant matrix, $B$ can be extracted from the summation leaving
    $\sum_{i=0}^{\infty} A^i z^{-i}$.
    For a system that is Schur-stable, this summation is a Neumann series that converges to
    $z( zI - A )^{-1}$ \cite[\S 3.4]{Peterson2012}.
    Substituting this into~\eqref{eq:predMat:trigPoly} then produces the matrix symbol $z(zI - A)^{-1}B$.
    The spectral coefficients \eqref{eq:predMat:specCoeff} are absolutely summable, so $P_{\Gamma}$ is in the Wiener class meaning that $P_{\Gamma} \in \Linf$ and is continuous and $2\pi$-periodic, leading to $P_{\Gamma} \in \Cpi$.
\end{proof}

For $\Gamma$, the resulting matrix symbol is a time-shifted version of the dynamical system.
 The assumption of Schur-stability of the system is necessary, since if the $A$ matrix were to have eigenvalues outside the unit circle, the Fourier series would no longer converge and the symbol would be unbounded.
 
One of the useful properties of Toeplitz matrices is the relation between the spectrum for $T_{n}$ and the spectrum of its matrix symbol.
 Specifically, the distribution of the spectrum of the matrix symbol evaluated on $\mathbb{T}$ is the same as the distribution of the spectrum for $T_{n}$ as $n \to \infty$, and the spectrum of $T_{n}$ will always be contained in the spectrum of its symbol.
 This means that we can utilize $P_{\Gamma}$ to find the distribution of the singular values of $\Gamma$ and bound them.
\begin{proposition}
    \label{prop:dense:predSVD}
    Let $G_{s}$ be a Schur-stable system predicted over a horizon of length $N$, then the following are true:
    \begin{enumerate}[label={(\alph*)},ref={\thetheorem(\alph*)}]
        \item $ \sigma_{min}( G_s ) \leq  \sigma(\Gamma) \leq \Hinfnorm{ G_s }$ 
            \label{thm:dense:predSVD:bounds}
        \item $ \underset{N \to \infty}{\lim} \kappa(\Gamma) = \kappa( G_{s}) $
            \label{thm:dense:predSVD:condLimit}
        \item $ \sigma(\Gamma) \approx \bigcup_{\omega \in \Omega} \sigma(G_s(e^{j \omega})) $
        
            with 
            $ \Omega \coloneqq \left\{ \omega: \omega=-\frac{\pi}{2} + \frac{2\pi}{N}i,\ i \in \mathbb{Z}_{[0,N-1]}\right\} $
            \label{thm:dense:predSVD:specEst}
    \end{enumerate}
\end{proposition}
\begin{proof}
    $\Gamma$ is lower-triangular, so the matrix $\ctrans{\Gamma} \Gamma$ is also Toeplitz with the matrix symbol $\ctrans{G_s} G_s \in \Cpi$ \cite[Lemma 4.5]{Gutierrez-Gutierrez2012_blockSurvey} since $\ctrans{z} z = 1$ for $z \in \mathbb{T}$.
    \begin{enumerate}[label={(\alph*)},ref={\alph*}]
        \item Since $\ctrans{\Gamma} \Gamma$ is Toeplitz with its symbol in $\Cpi$, its spectrum is upper bounded by
        $\sigma(\Gamma_n) \leq \sqrt{\lambda_{max}(\ctrans{G_{s}} G_{s})}$.
        Applying $\lambda(\ctrans{G_s} G_s) = \sigma(G_s)^{2}$ then produces the inequality $\sigma_1(\Gamma) \leq \sigma_{max}(G_s)$.
        Then note that the $H_{\infty}$ norm of a system is its largest singular value, giving the final inequality.
        
        The proof for the lower bound follows the same steps.
        \item Taking the limit of the condition number as $N\to \infty$ gives:
        \begin{equation*}
            \underset{N \to \infty}{\lim} \kappa(\Gamma) 
            = 
            \underset{N \to \infty}{\lim} \frac{\sigma_{max}(\Gamma)}{\sigma_{min}(\Gamma)}
            = \frac{\sigma_{max}(G_{s})}{\sigma_{min}(G_{s})}
            = \kappa( G_{s} ).
        \end{equation*}
        \item It is known that the singular values of matrix $\Gamma$ are related to the eigenvalues of $\ctrans{\Gamma} \Gamma$ through $\sigma(\Gamma)^2 = \lambda(\ctrans{\Gamma} \Gamma)$.
        From \cite{Miranda2000}, the eigenvalues of $\ctrans{\Gamma} \Gamma$ can be estimated as $N \to \infty$ by finding the eigenvalues of the matrix symbol $\ctrans{G_{s}} G_{s}$ as it is evaluated around the unit circle, e.g.
        \begin{equation}
            \label{eq:predMat:proofSvd}
            \lambda(\ctrans{\Gamma} \Gamma) \approx \bigcup_{\omega \in \Omega} \lambda{(\ctrans{G_{s}(e^{j \omega})} G_{s}(e^{j \omega}))}.
        \end{equation}
        Since the right-hand side of \eqref{eq:predMat:proofSvd} evaluates to a matrix at every point $\omega$, we can rewrite \eqref{eq:predMat:proofSvd} as
        $$
        \sigma(\Gamma)^2 \approx \bigcup_{\omega \in \Omega} \sigma{(G_{s}(e^{j \omega}))}^2.
        $$
        Taking the square root of both sides gives the final result.
    \end{enumerate}
\end{proof}

\begin{figure*}[t!]
    \begin{equation}
    \label{mat:equal:HQ}
    \bar{H}_{Q} \coloneqq
    \begin{bmatrix}
    \sum_{i=0}^{N-1} \ctrans{B} \ctrans{(A^{i})} Q A^{i} B   & \sum_{i=0}^{N-2} \ctrans{B} \ctrans{(A^{i+1})} Q A^{i} B & \sum_{i=0}^{N-3} \ctrans{B} \ctrans{(A^{i+2})} Q A^{i} B & \cdots & \ctrans{B} \ctrans{(A^{N-1})} Q B\\
    \sum_{i=0}^{N-2} \ctrans{B} \ctrans{(A^{i})} Q A^{i+1} B & \sum_{i=0}^{N-2} \ctrans{B} \ctrans{(A^{i})} Q A^{i} B   & \sum_{i=0}^{N-3} \ctrans{B} \ctrans{(A^{i+1})} Q A^{i} B & \cdots & \ctrans{B} \ctrans{(A^{N-2})} Q B\\
    \sum_{i=0}^{N-3} \ctrans{B} \ctrans{(A^{i})} Q A^{i+2} B & \sum_{i=0}^{N-3} \ctrans{B} \ctrans{(A^{i})} Q A^{i+1} B & \sum_{i=0}^{N-3} \ctrans{B} \ctrans{(A^{i})} Q A^{i} B   & \cdots &  \ctrans{B} \ctrans{(A^{N-3})} Q B\\
    \vdots & \vdots & \vdots & \ddots & \vdots \\
    \ctrans{B} Q A^{N-1} B & \ctrans{B} Q A^{N-2} B & \ctrans{B} Q A^{N-3} B & \cdots & \ctrans{B} Q B
    \end{bmatrix}
    \end{equation}
    \hrulefill
\end{figure*}

Proposition~\ref{prop:dense:predSVD} shows that the spectral properties of the prediction matrix $\Gamma$ can be related to the transfer function $G_{s}$ sampled around the unit circle.
 This allows for the singular values of $\Gamma$ to be both bounded and estimated by performing analysis on $G_{s}(\cdot)$ instead of actually forming $\Gamma$.




\subsection{Condensed Primal Hessian Matrix}
\label{sec:specRes:condHess}

The Hessian of the MPC problem formulation in~\eqref{eq:mpc:condMPC} can be split into four distinct parts
 \begin{equation}
    \label{eq:dense:primalHessian:MatrixSplitting}
     H_c \coloneqq H_Q + H_S + H_P + H_R
 \end{equation}
 where $H_Q$, $H_S$, $H_R$ and $H_P$ are the parts that contain the matrices $Q$, $S$, $R$ and $P$ respectively.
 In this work, we allow for arbitrary $Q$ and $R$ matrices, while presenting three specific cases for $S$ and $P$:
 \begin{enumerate}[label=Case \arabic*), ref=\arabic*, leftmargin=*]
     \item $P = Q$ and $S = 0$ \label{case:noPnoS}
     \item $P$ is the solution of the discrete-time Lyapunov equation $\trans{A} P A + Q = P$ and $S=0$ \label{case:Plyap}
     \item Arbitrary $S$ \label{case:withS}
 \end{enumerate}
 The specific selection of $P$ in Case~\ref{case:Plyap} is examined because it is commonly used to guarantee asymptotic stability of the closed-loop MPC controller \cite{Mayne2000_StabilitySurvey}.

\subsubsection*{Case \ref{case:noPnoS}}

Let the condensed Hessian matrix for this case be $H_{cQ}$.
 The assumption that $S = 0$ and $P = Q$ means that $H_S = 0$ and $H_P = 0$, making $H_{cQ} \coloneqq \bar{H}_{Q} + H_R$ where 
 $\bar{H}_{Q} \coloneqq \ctrans{\Gamma} \tilde{Q} \Gamma$ and $\tilde{Q} \coloneqq I_{N} \otimes Q$.
 Examining the structure of $\bar{H}_{Q}$ in \eqref{mat:equal:HQ}, it can be seen that the each diagonal possess a common structure, with the inconsistency being the summation end point for each entry.
 It turns out that $\bar{H}_{Q}$ is a Toeplitz matrix when examined as $N \to \infty$, and when it is added with $H_R$, the resulting matrix $H_{cQ}$ is a Toeplitz matrix represented by the generating symbol in Lemma~\ref{lem:HqSymbol}.

\begin{lemma}
    \label{lem:HqSymbol}
    Let $S=0$, $P = Q$ and $P_{\Gamma}$ be the matrix symbol from Lemma~\ref{lem:gammaSymbol} for a Schur-stable system, then the matrix $H_{cQ}$ is a Toeplitz matrix with the matrix symbol $P_{H_{cQ}} \in \Cpi$ where
    \begin{equation*}
        \label{eq:dense:PQ:Hq:genFunc}
        P_{H_{cQ}}(z) \coloneqq \ctrans{P_{\Gamma}(z)} Q P_{\Gamma}(z) + R \quad
        \forall z \in \mathbb{T}
    \end{equation*}
\end{lemma}
\begin{proof}
    $H_R \coloneqq I_N \otimes R$ is a Toeplitz matrix with symbol $P_R(z) \coloneqq R$.
    Under the assumptions $S = 0$ and $P = Q$, $H_S = 0$ and $H_P = 0$.
    Since $\Gamma$ is a lower-triangular matrix and $\ctrans{\Gamma}$ is an upper-triangular matrix, the product $\ctrans{\Gamma} \tilde{Q} \Gamma$ is Toeplitz with generating symbol
    $
    \ctrans{P_{\Gamma}} Q P_{\Gamma}
    $
    \cite[Lemma 4.5]{Gutierrez-Gutierrez2012_blockSurvey}.
    Additionally, Toeplitz structure is preserved over addition of two Toeplitz matrices, meaning matrix $H_{cQ}$ is then Toeplitz, with the symbol given in the Lemma.
\end{proof}

The matrix $H_{cQ}$ is the combination of three different Toeplitz matrices using addition and multiplication.
 The Toeplitz structure of a matrix is always preserved when adding/subtracting two Toeplitz matrices, and the new matrix symbol is simply the addition/subtraction of the two symbols.
 Unfortunately, the multiplication of Toeplitz matrices does not generally result in a Toeplitz matrix.
 However, when multiplied as $ \trans{L} T L$ (with $L$ lower-triangular), the Toeplitz structure is preserved and the new symbol is simply the multiplication of the symbols.

Since $H_{cQ}$ is Toeplitz with the symbol in Lemma~\ref{lem:HqSymbol}, we can estimate and bound the eigenvalues $H_{cQ}$ using the symbol.
\begin{theorem}
    \label{thm:dense:PQ:hessEig}
    Let $H_{cQ}$ be the condensed Hessian matrix for a Schur-stable system predicted over a horizon of length $N$ with $S = 0$, $P=Q$ and the matrix symbol $P_{H_{cQ}}$ given in Lemma~\ref{lem:HqSymbol}, then the following are true:
    \begin{enumerate}[label={(\alph*)},ref={\thetheorem(\alph*)}]
        \item $ \lambda_{min}( P_{H_{cQ}} ) \leq  \lambda(H_{cQ}) \leq \lambda_{max}( P_{H_{cQ}}) $
        \label{thm:dense:PQ:hessEig:bounds}
        \item $ \underset{N \to \infty}{\lim} \kappa(H_{cQ}) = \kappa( P_{H_{cQ}} ) $
        \label{thm:dense:PQ:hessEig:condLimit}
        \item $ \lambda( H_{cQ} ) \approx \bigcup_{\omega \in \Omega} \lambda(P_{H_{cQ}}(e^{j \omega})) $
        
        with 
        $ \Omega \coloneqq \left\{ \omega: \omega=-\frac{\pi}{2} + \frac{2\pi}{N}i, \quad i \in \mathbb{Z}_{[0,N-1]}\right\} $
        \label{thm:dense:PQ:hessEig:specEst}
    \end{enumerate}
\end{theorem}
\begin{proof}
    \begin{enumerate}[label={(\alph*)},ref={\thetheorem(\alph*)}]
        \item The spectrum of a Toeplitz matrix with its symbol in $\Cpi$ is bounded by the extremes of the spectrum of its symbol \cite[Theorem 4.4]{Gutierrez-Gutierrez2012_blockSurvey}.
        \item Note that $H_{cQ}$ is a Hermitian matrix, which means that it is also normal \cite[\S 4.1]{Horn2013}. Since it is both normal and positive semi-definite, $\sigma(H_{cQ}) = \lambda(H_{cQ})$ \cite[\S 3.1]{Horn1994}, making the condition number become $\kappa(H_{cQ}) = \frac{\lambda_{n}(H_{cQ})}{\lambda_{1}(H_{cQ})}$. Taking the limit of both sides in conjunction with the spectral bounds from part~(a) gives
        \begin{equation*}
        \underset{N \to \infty}{\lim} \kappa(H_{cQ}) 
        = 
        \kappa( P_{H_{cQ}} ).
        \end{equation*}
        \item The spectrum of a Toeplitz matrix can be estimated from its symbol using the techniques in \cite{Miranda2000}, giving this result.
    \end{enumerate}
\end{proof}

Since the eigenvalues for any finite-size Toeplitz matrix are guaranteed to be inside the spectrum of its matrix symbol, the bounds given in Theorem~\ref{thm:dense:PQ:hessEig} are horizon-independent.
 As the matrix size grows, the bounds in Theorem~\ref{thm:dense:PQ:hessEig:bounds} become tight; so equality occurs for long prediction horizons.
 The horizon at which equality occurs differs for every system, with some reaching it in short horizons (e.g.\ System~\ref{sys:jonesMorari} with equality for $N > 40$), while others require very long horizons (e.g.\ System~\ref{sys:msd} requiring $N \gg 1000$).
 This is illustrated in Figure~\ref{fig:specRes:primalHessSpec}.
 
Since the matrix symbol $P_{H_{cQ}}$ takes a complex number on the unit circle as its argument, we can view it as the discrete-time transfer function of a dynamical system.
 This means that the extremal eigenvalues can be found through the $H_{\infty}$ norm of the system and its inverse (which will always exist since $H_{cQ}$ is positive-definite).
 Additionally, the spectrum of $H_{cQ}$ can be estimated through the singular values of the discrete-time system $P_{H_{cQ}}$ since $H_{cQ}$ is Hermitian.

\begin{figure*}[t!]
    \begin{equation}
    \label{mat:Lyap:HP}
    H_{P2} \coloneqq
    \begin{bmatrix}
    \ctrans{B} \ctrans{(A^{N})} P A^{N} B   & \ctrans{B} \ctrans{(A^{N})} P A^{N-1} B   & \ctrans{B} \ctrans{(A^{N})} P A^{N-2} B   & \cdots & \ctrans{B} \ctrans{(A^{N})} P A B\\
    \ctrans{B} \ctrans{(A^{N-1})} P A^{N} B & \ctrans{B} \ctrans{(A^{N-1})} P A^{N-1} B & \ctrans{B} \ctrans{(A^{N-1})} P A^{N-2} B & \cdots & \ctrans{B} \ctrans{(A^{N-1})} P A B\\
    \ctrans{B} \ctrans{(A^{N-2})} P A^{N} B & \ctrans{B} \ctrans{(A^{N-2})} P A^{N-1} B & \ctrans{B} \ctrans{(A^{N-2})} P A^{N-2} B & \cdots & \ctrans{B} \ctrans{(A^{N-2})} P A B\\
    \vdots                                  & \vdots                                    & \vdots                          & \ddots & \vdots \\
    \ctrans{B} \ctrans{A} P A^{N} B       & \ctrans{B} \ctrans{A} P A^{N-1} B       & \ctrans{B} \ctrans{A} P A^{N-2} B       & \cdots & \ctrans{B} \ctrans{A} P A B
    \end{bmatrix}
    \end{equation}
    \hrulefill
\end{figure*}

The results in Theorem~\ref{thm:dense:PQ:hessEig} give the spectral properties of $H_{cQ}$ for arbitrary $Q$ and $R$ matrices.
 If we were to constrain the matrices to be $Q = I$ (or $Q = C^{T} C$ for systems with an output mapping) and $R = \rho I$, then the results presented in \cite[Corrollary 11.5.2]{Goodwin2005} will be recovered.

\subsubsection*{Case \ref{case:Plyap}}

We now examine the Hessian matrix that results from choosing a terminal cost matrix such that $P$ is the solution to the discrete-time Lyapunov equation.
 The matrix splitting for $H_c$ in~\eqref{eq:dense:primalHessian:MatrixSplitting} does not generally have nice properties when $P \ne Q$.
 However, with $P$ chosen as the solution to $ \trans{A} P A + Q = P$, the matrix splitting becomes $H_{cP} \coloneqq H_{cQ} + H_{P2}$ where $H_{cQ}$ is the Hessian from Case~\ref{case:noPnoS} and $H_{P2}$ is the matrix given in~\eqref{mat:Lyap:HP}.
 The structure of $H_{cP}$ is such that as $N\to \infty$, its eigenvalues converge to the eigenvalues of $H_{cQ}$.
 This means the results in Theorem~\ref{thm:dense:PQ:hessEig} can be used to bound the spectrum of $H_{cP}$.
\begin{theorem}
    \label{thm:dense:PLyap:hessEig}
    Let the system $G_{s}$ be Schur-stable with the state space matrices $(A, B, I, 0)$, and $P_{H_{cQ}}$ be defined in Lemma~\ref{lem:HqSymbol}.
    If the terminal weighting matrix $P$ is chosen as the solution to $\trans{A} P A + Q = P$, then the spectrum of the condensed Hessian matrix $H_{cP}$ has the following properties:
    \begin{enumerate}[label={(\alph*)},ref={\thetheorem(\alph*)}]
        \item $ \lambda_{min}(P_{H_{cQ}}) \leq  \lambda(H_{cP}) \leq \lambda_{max}(P_{H_{cQ}}) $ 
        \label{thm:dense:PLyap:hessEig:eigBound}
        \item $ \underset{N \to \infty}{\lim} \kappa(H_{cP}) = \kappa(P_{H_{cQ}}) $
        \label{thm:dense:PLyap:hessEig:condBound}
    \end{enumerate}
\end{theorem}
\begin{proof}
    Since $P$ is the solution to the discrete-time Lyapunov equation, $P$ can be written as the infinite sum
    $ P = \sum_{i=0}^{\infty} \trans{(A^{i})} Q A^{i} $.
    Substituting this into $H_{P2}$ produces a matrix which has the same entries as $H_{cQ}$, except with the starting/ending points of the summations changed.
    In this new matrix, the summations start where the summations in $H_{cQ}$ stop, and end at $\infty$ (e.g.\ the upper-left corner starts at $N$ and goes to $\infty$, while the lower-right corner starts at $1$ and goes to $\infty$).
    This means that adding $H_{P2}$ to $H_{cQ}$ changes the summations in $H_{cQ}$ so that all of them now go from $0$ to $\infty$.
    This results in $H_{cP}$ becoming a Toeplitz matrix, with the same matrix symbol as $H_{cQ}$ in Lemma~\ref{lem:HqSymbol}.
    Then for $N \to \infty$, the spectrum of $H_{cP}$ converges to the spectrum of $H_{cQ}$.
\end{proof}

The results in Theorem~\ref{thm:dense:PLyap:hessEig} are intuitive, since the terminal cost $\trans{x_{N}} P x_{N}$ is designed to capture the value of the cost after the prediction horizon and should therefore disappear when going to infinite horizons.
 For the CLQR formulation~\eqref{eq:mpc:linMPC}, selecting $P$ as the solution to the discrete-time Lyapunov equation will exactly capture the cost's tail.
 When condensed into~\eqref{eq:mpc:condMPC}, $H_{P2}$ extends the summations in each entry of $H_{cQ}$ to infinity in order to capture this tail.
 This means that as $N \to \infty$, the effect of $H_{P2}$ on $H_{cQ}$ will diminish, until eventually $H_{P2}$ vanishes.
 This leads to the horizon-independent result in Theorem~\ref{thm:dense:PLyap:hessEig} that the extremal eigenvalues and condition number of $H_{cP}$ will be the same as those for $H_{cQ}$.
 Note though that the finite-horizon spectrum of $H_{cP}$ may be different from that of $H_{cQ}$, as shown by the difference between them (represented by triangles and circles respectively) at low horizons in Figure~\ref{fig:specRes:primalHessSpec:sys1}.


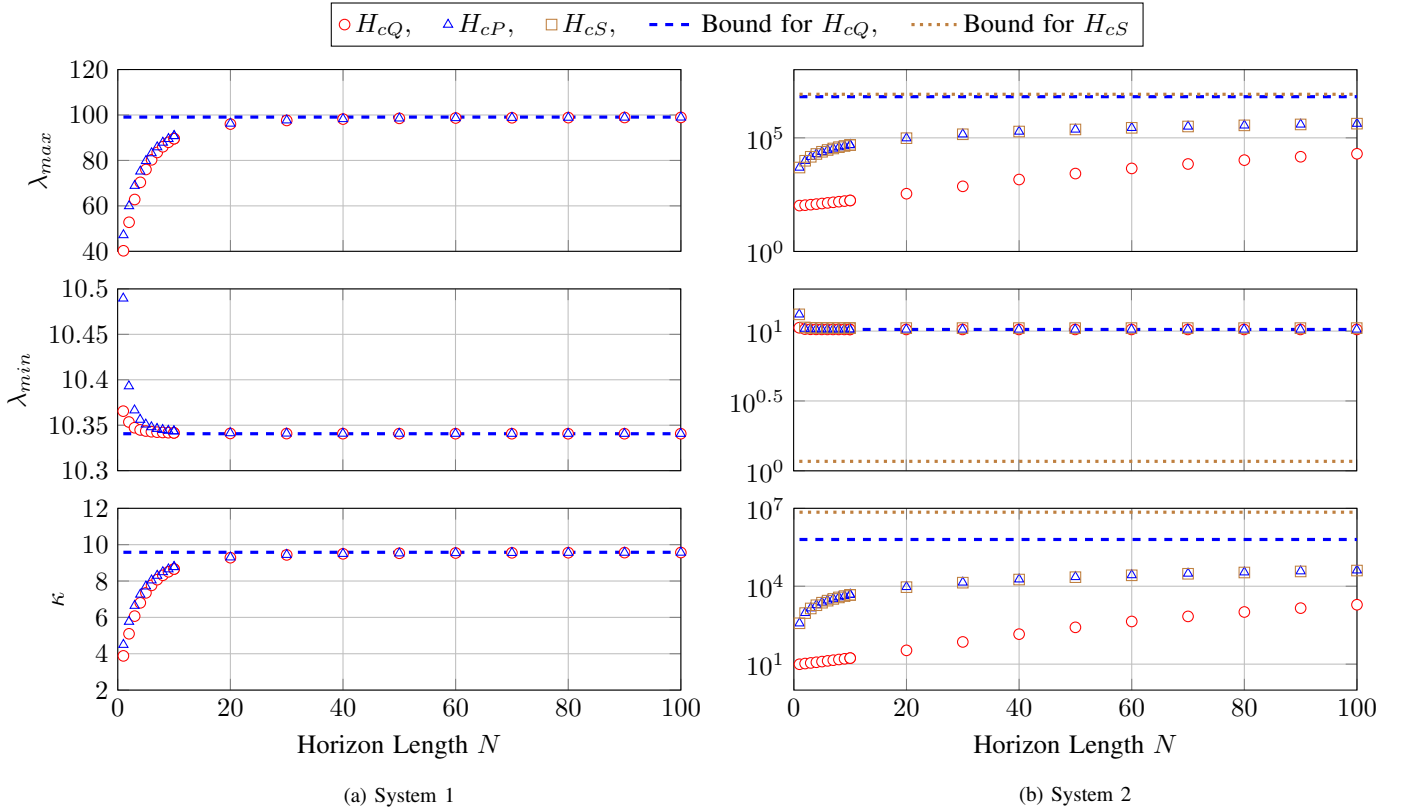
\begin{figure*}[tb]
    \centering
    
    \pgfplotsset{Q_inst/.style={only marks, red, every mark/.append style={solid, fill=white}, mark = *}}
    \pgfplotsset{Q_limit/.style={dashed, very thick, blue}}
    \pgfplotsset{Lyap_inst/.style={only marks, blue, every mark/.append style={solid, fill=white}, mark = triangle}}
    \pgfplotsset{S_inst/.style={only marks, brown, every mark/.append style={solid, fill=white}, mark = square}}
    \pgfplotsset{S_limit/.style={dotted, very thick, brown}}
    
    \begin{tikzpicture}
        \begin{groupplot}[group style = {group name=plots, group size = 2 by 3, horizontal sep=1.5cm, vertical sep=0.5cm}]
        
        \pgfplotstableread[col sep=comma]{figures/data/spectrum_Primal_JonesMorari_Inst.csv}{\JMinstData}
        \pgfplotstableread[col sep=comma]{figures/data/spectrum_Primal_JonesMorari_Asymp.csv}{\JMasympData}
        \pgfplotstableread[col sep=comma]{figures/data/spectrum_Primal_MassSpring_Inst.csv}{\MSDinstData}
        \pgfplotstableread[col sep=comma]{figures/data/spectrum_Primal_MassSpring_Asymp.csv}{\MSDasympData}

        \nextgroupplot[xmin   = 0,
                       xmax   = 100,
                       ymin   = 40,
                       ymax   = 120,
                       grid   = major,
                       ylabel = {$\lambda_{max}$ },
                       xticklabels=\empty,
                       height = 4cm,
                       width  = 0.5\textwidth]
    
            \addplot[Q_inst] table [x=Horizon, y=maxE_Q] {\JMinstData};
            
            \addplot[Q_limit] table [x=Horizon, y=maxE_Q] {\JMasympData};
            
            \addplot[Lyap_inst] table [x=Horizon, y=maxE_Lyap] {\JMinstData};
            

        \nextgroupplot[xmin   = 0,
                       xmax   = 100,
                       ymin   = 1e0,
                       ymax   = 1e8,
                       ymode  = log,
                       grid   = major,
                       xticklabels=\empty,
                       height = 4cm,
                       width  = 0.5\textwidth,
                       legend cell align=left,
                       legend columns=-1,
                       legend to name=plotleg]
        
            \addplot[Q_inst] table [x=Horizon, y=maxE_Q] {\MSDinstData};
            \addlegendentry{$H_{cQ},\quad$}
            
            \addplot[Lyap_inst] table [x=Horizon, y=maxE_Lyap] {\MSDinstData};
            \addlegendentry{$H_{cP},\quad$}
            
    
            \addplot[S_inst] table [x=Horizon, y=maxE_S] {\MSDinstData};
            \addlegendentry{$H_{cS},\quad$}

            \addplot[Q_limit] table [x=Horizon, y=maxE_Q] {\MSDasympData};
            \addlegendentry{Bound for $H_{cQ},\quad$}

            \addplot[S_limit] table [x=Horizon, y=maxE_S] {\MSDasympData};
            \addlegendentry{Bound for $H_{cS}$}
    
        \nextgroupplot[xmin   = 0,
                       xmax   = 100,
                       ymin   = 10.3,
                       ymax   = 10.5,
                       grid   = major,
                       xticklabels=\empty,
                       ylabel = {$\lambda_{min}$},
                       height = 4cm,
                       width  = 0.5\textwidth]
            
                \addplot[Q_inst] table [x=Horizon, y=minE_Q] {\JMinstData};
                
                \addplot[Lyap_inst] table [x=Horizon, y=minE_Lyap] {\JMinstData};
            
                \addplot[Q_limit] table [x=Horizon, y=minE_Q] {\JMasympData};
                

        \nextgroupplot[xmin   = 0,
                       xmax   = 100,
                       ymin   = 1,
                       ymax   = 20,
                       ymode  = log,
                       grid   = major,
                       xticklabels=\empty,
                       height = 4cm,
                       width  = 0.5\textwidth]
        
            \addplot[Q_inst] table [x=Horizon, y=minE_Q] {\MSDinstData};
            
            \addplot[Lyap_inst] table [x=Horizon, y=minE_Lyap] {\MSDinstData};
            
            \addplot[S_inst] table [x=Horizon, y=minE_S] {\MSDinstData};
            
            \addplot[Q_limit] table [x=Horizon, y=minE_Q] {\MSDasympData};


            \addplot[S_limit] table [x=Horizon, y=minE_S] {\MSDasympData};

        \nextgroupplot[xmin   = 0,
                       xmax   = 100,
                       ymin   = 2,
                       ymax   = 12,
                       grid   = major,
                       xlabel = {Horizon Length $N$},
                       ylabel = {$\kappa$},
                       height = 4cm,
                       width  = 0.5\textwidth]
           
                \addplot[Q_inst] table [x=Horizon, y=condNum_Q] {\JMinstData};
                
                \addplot[Lyap_inst] table [x=Horizon, y=condNum_Lyap] {\JMinstData};
                
                \addplot[Q_limit] table [x=Horizon, y=condNum_Q] {\JMasympData};
                
            
        \nextgroupplot[xmin   = 0,
                       xmax   = 100,
                       ymin   = 1e0,
                       ymax   = 1e7,
                       ymode  = log,
                       grid   = major,
                       xlabel = {Horizon Length $N$},
                       height = 4cm,
                       width  = 0.5\textwidth]
            
            \addplot[Q_inst] table [x=Horizon, y=condNum_Q] {\MSDinstData};
            
            \addplot[Lyap_inst] table [x=Horizon, y=condNum_Lyap] {\MSDinstData};
            
            \addplot[S_inst] table [x=Horizon, y=condNum_S] {\MSDinstData};            
            
            \addplot[Q_limit] table [x=Horizon, y=condNum_Q] {\MSDasympData};
            
                                    
            \addplot[S_limit] table [x=Horizon, y=condNum_S] {\MSDasympData};
    
        \end{groupplot}
        
        \node at ($(plots c1r1.north west)!0.5!(plots c2r1.north east)+(0cm,0.6cm)$) {\pgfplotslegendfromname{plotleg}};
        
        
        \node[below=1.1cm of plots c1r3.south, name=label1, anchor=center] {
            \subfloat[System~\ref{sys:jonesMorari}]
              {\parbox[c][0.5cm][c]{3cm}{\hfill\label{fig:specRes:primalHessSpec:sys1}}}
         };
         
        \path let \p{lab}=(label1.center), \p{gra}=(plots c2r3.south) in
         node[name=label2]
          at (\x{gra}, \y{lab}) {
            \subfloat[System~\ref{sys:msd}]
              {\parbox[c][2pt][c]{3cm}{\hfill\label{fig:specRes:primalHessSpec:sys2}}}%
         };
    \end{tikzpicture}

    \caption{Spectral properties of the condensed primal Hessian matrix. The lines represent the bounds computed using the results in Section~\ref{sec:specRes:condHess}, and the markers represent the values of the condensed matrix at that horizon.}
    \label{fig:specRes:primalHessSpec}
\end{figure*}

\subsubsection*{Case \ref{case:withS}}
 
We now introduce a state-input cross-term weighting matrix $S$ to problem~\eqref{eq:mpc:linMPC}, and refer to the resulting Hessian matrix as $H_{cS}$.
 Unfortunately, $H_{S}$ is not Toeplitz since $\bar{S}$ contains a $0$ matrix instead of $S$ in the lower-right corner; which when multiplied with $\Gamma$ produces a row/column of zeros on the bottom/right of the matrix $H_{S}$.
 
To overcome this, we split $H_{cS}$ into two components, a nominal matrix $H_{n}$ and a correction matrix $H_{e}$, such that $H_{cS} \coloneqq H_{n} - H_{e}$.
 We let the nominal matrix be $H_{n} \coloneqq H_{cQ} + \bar{H}_{S}$, where
 $
 \bar{H}_S \coloneqq \trans{(I_N \otimes S)} \Gamma + \trans{\Gamma} (I_N \otimes S)
 $.
 This adds in an $S$ weighting term on the final state, which makes $\bar{H}_S$ Toeplitz with the matrix symbol
 $$
 P_{\bar{H}_S}(z) \coloneqq \trans{S} P_{\Gamma}(z) + \ctrans{P_{\Gamma}}(z) S \qquad \forall z \in \mathbb{T}.
 $$
 This leads to $H_{n}$ being Toeplitz as well with matrix symbol $P_{H_n} \in \Cpi$,
 \begin{equation}
    \label{eq:spec:sTerm:symHn}
     P_{H_n}(z) = P_{H_{cQ}}(z) + P_{\bar{H}_S}(z) \qquad \forall z \in \mathbb{T}.
 \end{equation}
 
The additional weighting term introduced in $H_{n}$ is then corrected for by subtracting the Hermitian matrix $H_{e}$, where
 $H_{e} \coloneqq \trans{S_{c}} \Gamma + \trans{\Gamma} S_{c}$ with 
 $S_{c} \coloneqq \begin{bmatrix}
 I_{N-1} \otimes 0 & 0\\
 0 & S
 \end{bmatrix}$.
 To understand how $H_{e}$ affects $H_{n}$, we first analyze the spectrum of $H_{e}$.
 \begin{lemma}
     \label{lem:dense:Sterm:HeRank}
     Let $A \in \mathbb{R}^{n \times n}$, $B \in \mathbb{R}^{n \times m}$ and $W_{c}$ be the state transition matrix, input matrix and the controllability Gramian respectively for the Schur-stable system $G_{s}$. If $S \neq 0$, then the rank of the matrix $H_{e}$ is at most $2m$, and for $N \to \infty$ its $2m$ non-zero eigenvalues are the $2m$ eigenvalues of 
     \begin{equation*}
         \label{eq:specRes:Sterm:eigenMatrix}
         U \coloneqq
         \begin{bmatrix}
             \trans{B} S & I \\
             \trans{S} W_{c} S & \trans{S} B
         \end{bmatrix}.
     \end{equation*}
 \end{lemma}
 \begin{proof}
    $H_{e}$ is the outer product $v \trans{u}$ of $v,u \in \mathbb{R}^{Nm \times 2m}$ defined as 
    \begin{align*}
    v &\coloneqq \trans{
        \begin{bmatrix}
        \trans{S} A^{N} B & \trans{S} A^{N-1} B & \cdots & \trans{S} A B & \trans{S} B\\
        0 & 0 & \cdots & 0 & I
        \end{bmatrix}}, \\
    u &\coloneqq \trans{
        \begin{bmatrix}
        0 & 0 & \cdots & 0 & I \\
        \trans{S} A^{N} B & \trans{S} A^{N-1} B & \cdots & \trans{S} A B & \trans{S} B
        \end{bmatrix}}.
    \end{align*}
    The rank of an outer product matrix can be no larger than the smallest rank of the component matrices, and the rank of both $u$ and $v$ is $\leq 2m$.
    The non-zero eigenvalues of $v \trans{u}$ are the same as the eigenvalues of $\trans{u} v$ \cite[Example~1.3.23]{Horn2013}, with
    \begin{equation*}
    \trans{u} v =
    \begin{bmatrix}
    \trans{B} S & I \\
    \sum_{k=0}^{N} \trans{S} A^{k} B \trans{B} \trans{(A^{k})} S & \trans{S} B
    \end{bmatrix}
    = U.
    \end{equation*}
    As $N \to \infty$, the summation converges to the controllability Gramian of the system $G_{s}$ \cite[\S 6.6]{Chen1999_LinearSystemsBook}, which makes the lower-left corner of the matrix $U$ become $\trans{S} W_{c} S$.
 \end{proof}
    
The results presented in Lemma~\ref{lem:dense:Sterm:HeRank} show that the correction matrix $H_{e}$ has a finite and low rank independent of prediction horizon.
 Additionally, the limit points for the eigenvalues of $H_{e}$ as $N \to \infty$ can be computed by finding the eigenvalues of the $2m \times 2m$ matrix $U$, which is independent of the horizon length.
 Knowledge of the eigenvalues of $H_{e}$ then allows for the eigenvalues of the Hessian matrix $H_{cS}$ to be bounded.
\begin{theorem}
    \label{thm:dense:Sterm:hessEig}
    Let the system $G_{s}$ be Schur-stable with $P_{H_{n}}$ from \eqref{eq:spec:sTerm:symHn} and $U$ from Lemma~\ref{lem:dense:Sterm:HeRank}.
    Let 
    $ \gamma~\coloneqq~\lambda_{max}(P_{H_{n}}) $, 
    $ \beta~\coloneqq~\lambda_{min}(P_{H_{n}}) $,
    $ \eta~\coloneqq~\lambda_{max}( U ) $,
    $ \nu~\coloneqq~\lambda_{min}( U ) $.
    If $S \neq 0$, then the spectrum of the condensed Hessian matrix $H_{cS}$ has the following properties:
    \begin{enumerate}[label={(\alph*)},ref={\thetheorem(\alph*)}]
        \item $ \max\{0, \beta - \eta\} \leq  \lambda(H_{cS}) \leq \gamma - \nu$ 
        \label{thm:dense:Sterm:hessEig:eigBound}
        \item $ \underset{N \to \infty}{\lim} \kappa(H_{cS}) \leq 
        \begin{cases}
        \frac{\gamma - \nu}{\beta - \eta} & \text{if}~\beta > \eta\\
        \infty & \text{otherwise}
        \end{cases}  $
        \label{thm:dense:Sterm:hessEig:condBound}
    \end{enumerate}
\end{theorem}
\begin{proof}
    Note that $H_{cS} = H_{n} - H_{e}$ can be viewed as the addition of the negation of $H_{e}$.
    Negating a matrix will negate all the eigenvalues, and consequently reverse their order.
    Since both $H_{n}$ and $H_{e}$ are Hermitian, the eigenvalues of $H_{cS}$ can be bounded by \cite[Fact 5.12.2]{Bernstein2009}
    \begin{align*}
    \lambda_{min}(H_{n}) - \lambda_{max}(H_{e}) &\leq \lambda_{min}(H_{cS}), \\
    \lambda_{max}(H_{cS}) &\leq \lambda_{max}(H_{n}) - \lambda_{min}(H_{e}),
    \end{align*}
    which gives the inequalities in part~(a).
    No a priori bounds are provided for the value of $\eta$, so it is possible that $\beta - \eta < 0$.
    However it is given that the Hessian matrix is positive definite, so when $\beta - \eta < 0$ the lower bound is set to $0$.
    The condition number in part (b) follows from applying the bounds in part (a) to the definition of the condition number.
\end{proof}
    
The results presented in Theorem~\ref{thm:dense:Sterm:hessEig} provide horizon-independent bounds for the extremal eigenvalues of the primal Hessian matrix with $S$ present.
 The results in Theorem~\ref{thm:dense:PQ:hessEig:bounds} can provide the values for $\gamma$ and $\beta$ when $P=Q$ and for when $P$ is the solution to the discrete-time Lyapunov equation.
 Horizon-independent values for $\eta$ and $\nu$ can be computed using Lemma~\ref{lem:dense:Sterm:HeRank}.
 
The results in Theorem~\ref{thm:dense:Sterm:hessEig} are conservative bounds on the spectrum.
 This can be seen in Figure~\ref{fig:specRes:primalHessSpec:sys2}, where the bound on $\lambda_{min}$ is much lower than the actual computed eigenvalues.
 It is possible for the bound on the condition number to go infinite if $\lambda_{max}(H_{e}) \geq \lambda_{min}(H_n)$, since then the lower bound on the eigenvalue will be $0$ even though the actual Hessian $H_{cS}$ remains positive definite.

An alternative method presented in \cite{Anderson1982_linearOptimalBook} to handle a non-zero $S$ matrix is to transform the problem into one with the system given by $(\tilde{A}, \tilde{B})$ and weight matrices $(\tilde{Q}, \tilde{R}, \tilde{S})$ with
\begin{align*}
    \tilde{A} &\coloneqq A - B R^{-1} \trans{S}, & &\tilde{B} = B, \\
    \tilde{Q} &\coloneqq Q - S R^{-1} \trans{S}, & &\tilde{R} = R, & &\tilde{S}=0.
\end{align*}
Note that the Schur-stability assumption must now hold for $\tilde{A}$, which is not true in general.
This means a system that was Schur-stable before the transformation may lose its stability when transformed, and then must be pre-stabilized before the results from Cases~\ref{case:noPnoS} and~\ref{case:Plyap} can be used.

\subsection{Condensed Constraint Matrix}
\label{sec:specRes:condConstraint}

We now turn our focus to the constraints in~\eqref{eq:mpc:linMPC}.
 When both state and input constraints are included, the condensed constraint matrix~\eqref{eq:mpc:cond:con} is a lower-triangular Toeplitz matrix.
 \begin{lemma}
    \label{lem:specRes:conSymbol}
    For a Schur-stable system with prediction matrix $\Gamma$, the condensed constraint matrix $G$ is Toeplitz with the matrix symbol $P_{G} \in \Cpi$ with
    \begin{equation*}
    \label{eq:dense:con:genFunc}
    P_{G}(z) \coloneqq \begin{bmatrix}
    D P_{\Gamma}(z)\\
    E \\
    \end{bmatrix}
    \qquad
    \forall z \in \mathbb{T},
    \end{equation*}
    where $P_{\Gamma}(z)$ is the matrix symbol of $\Gamma$ given in Lemma~\ref{lem:gammaSymbol}.
 \end{lemma}
 \begin{proof}
    We start with the fact that
    $
    G = \bar{D} \Gamma + \bar{E}.
    $
    Using the definitions of $\bar{D}$ and $\bar{E}$, it is obvious they are Toeplitz with symbols
    \begin{equation*}
        P_{D}(z) \coloneqq
         \begin{bmatrix}
          D\\
          0_{l \times n} \\
         \end{bmatrix}, \quad
        P_{E}(z) \coloneqq
         \begin{bmatrix}
          0_{j \times m}\\
          E
         \end{bmatrix}
         \quad
         \forall z \in \mathbb{T}
    \end{equation*}
    respectively.
    The product $\bar{D} \Gamma$ is Toeplitz since $\bar{D}$ is diagonal.
    Combining the matrix symbols together leads to $P_{G}$.
 \end{proof}

Since $G$ is a Toeplitz matrix, we can directly relate the singular value distribution of $G$ to the singular values of $P_{G}$ in the same manner as in Section~\ref{sec:specRes:condPred}.
\begin{lemma}
    \label{lem:const:constSpec}
    Let $P_{G}$ be as defined in Lemma~\ref{lem:specRes:conSymbol}. If the system is predicted with a horizon of length $N$, then:
    \begin{enumerate}[label={(\alph*)},ref={\thetheorem(\alph*)}]
        \item $ \sigma_{min}( P_{G} ) \leq  \sigma(G) \leq \sigma_{max}(P_{G})$ %
        \label{thm:const:constSpec:bounds}
        \item $ \underset{N \to \infty}{\lim} \kappa(G) = \kappa( P_{G}) $
        \label{thm:const:constSpec:condLimit}
        \item $ \sigma(G) \approx \bigcup_{\omega \in \Omega} \sigma(P_{G}(e^{j \omega})) $
        
        with 
        $ \Omega \coloneqq \left\{ \omega: \omega=-\frac{\pi}{2} + \frac{2\pi}{N}i, \ i \in \mathbb{Z}_{[0,N-1]}\right\} $
        \label{thm:const:constSpec:specEst}
    \end{enumerate}
\end{lemma}
\begin{proof}
    This proof is similar to Proposition~\ref{prop:dense:predSVD}'s proof.
\end{proof}

%

\subsection{Dual Hessian Matrix}
\label{sec:specRes:dualHess}

In this section, we derive the spectral properties of the dual Hessian matrix $H_d$ for two distinct cases:
 \begin{enumerate}[ref=\arabic*, leftmargin=*] 
     \item $H_{c}$ is arbitrary (e.g.\ Case~\ref{case:withS} from Section~\ref{sec:specRes:condHess}) \label{case:dual:noInput}
     \item $H_{c}$ is Toeplitz (e.g.\ Cases~\ref{case:noPnoS} and~\ref{case:Plyap} from Section~\ref{sec:specRes:condHess}) \label{case:dual:bothSI}
 \end{enumerate}

\subsubsection{$H_{c}$ is arbitrary}
\label{sec:specRes:dualHess:toepHess}

For problems where $H_{c}$ is not Toeplitz, the resulting dual Hessian matrix $H_{d}$ will also be non-Toeplitz.
 This means that a relationship between a matrix symbol and the spectrum of $H_d$ cannot be derived.
 We can however still place an upper bound on the spectrum of $H_d$.
 \begin{proposition}
    \label{prop:dense:dual:arbMatrixUB}
    If $P_{G}$ is defined as in Lemma~\ref{lem:specRes:conSymbol} and $H_{c}$ is the primal Hessian matrix, then
    \begin{equation*}
        \Twonorm{H_d} \leq \frac{\left( \sigma_{max}(P_{G}) \right)^2}{\lambda_{min}(H_c)}.
    \end{equation*}
 \end{proposition}
 \begin{proof}
    Combining the triangle inequality
    \begin{equation*}
        \Twonorm{H_d} \leq \Twonorm{G} \Twonorm{H_{c}^{-1}} \Twonorm{G}
    \end{equation*}
    and the fact that $\Twonorm{H_{c}^{-1}} = \frac{1}{\lambda_{min}(H_c)}$ together with Lemma~\ref{lem:const:constSpec} gives the result.
 \end{proof}

The result in Proposition~\ref{prop:dense:dual:arbMatrixUB} creates an upper bound for the spectrum of $H_d$ using the maximum singular value of the constraint matrix $G$ and the minimum eigenvalue of the primal Hessian matrix $H_c$.
 This result holds true for any Hessian matrix $H_{c}$, but is most applicable for Case~\ref{case:withS} in Section~\ref{sec:specRes:condHess} since the introduction of the $S$ term disrupts the Toeplitz structure.

A non-zero lower bound for the spectrum of $H_d$ does not in general exist, since $H_d$ can be rank-deficient depending on the constraint set $G$.
 \begin{proposition}
    \label{prop:spec:dual:arbMatixRank}
    Let $G \in \mathbb{R}^{Nl \times Nm}$ be the condensed constraint matrix for a horizon of length $N$, then
    $\text{Rank}(H_{d}) = \text{Rank}(G)$.
 \end{proposition}
 \begin{proof}
    Let $H_c$ be the positive definite primal Hessian matrix as defined in Section~\ref{sec:specRes:condHess}.
    Recall that $H_{d} = G H_{c}^{-1} \trans{G}$.
    It is known from \cite[Corollary~2.5.10]{Bernstein2009} that for matrix multiplication of two matrices $A \in \mathbb{R}^{x \times y}$ and $B \in \mathbb{R}^{y \times z}$ with ranks $a$ and $b$ respectively, the rank of $AB$ is
    \begin{equation}
        \label{eq:spec:dual:proof:rank}
        a + b - y \leq \text{Rank}(AB) \leq \min\{ a, b\}.
    \end{equation}
    
    Begin by examining the product $M \coloneqq GH_{c}^{-1}$, which leads to $M \in \mathbb{R}^{ N(j+l) \times Nm}$
    Since $H_{c}$ is positive definite, its inverse exists and is also full rank, meaning $\text{Rank}(H_{c}^{-1})=Nm$.
    Additionally, $\text{Rank}(G) \leq Nm$ since one dimension of $G$ is fixed at $Nm$.
    This means that both sides of~\eqref{eq:spec:dual:proof:rank} become $\text{Rank}(G)$, making $\text{Rank}(M) = \text{Rank}(G)$.
    A similar process can be followed for the product $M\trans{G}$, to get the final result.
 \end{proof}

Proposition~\ref{prop:spec:dual:arbMatixRank} shows that the rank of $H_d$ is determined by the constraint set.
 If there are more constraints than inputs (e.g.\ $l + j > m$), then the dual Hessian will be rank deficient, and therefore be positive semi-definite.

\subsubsection{$H_{c}$ is Toeplitz}
\label{sec:specRes:dualHess:toep}

If the MPC problem~\eqref{eq:mpc:linMPC} has a Toeplitz Hessian matrix (e.g.\ Case~\ref{case:noPnoS} or~\ref{case:Plyap} in Section~\ref{sec:specRes:condHess} or the transformed problem to remove $S$), then the eigenvalue distribution of $H_{d}$ can be estimated from the eigenvalues of its matrix symbol.
 To do this, we first note that the dual Hessian has the same non-zero eigenvalues as the matrix $H_{d1} \coloneqq H_{c}^{-1} \trans{G} G$.
\begin{lemma}
    \label{lem:dual:toep:eigEquivalence}
    Let $H_{c}$ and $G$ be from \eqref{eq:mpc:condMPC}.
    The non-zero eigenvalues of the Hessian $H_d$ in \eqref{eq:mpc:dual} are the same as the eigenvalues of $H_{d1} \coloneqq H_{c}^{-1} \trans{G} G$.
\end{lemma}
\begin{proof}
    This result follows directly from the fact that $H_{d}$ can be viewed as the outer product $u \trans{v}$ of the matrices $u \coloneqq G$ and $v \coloneqq H_{c}^{-1} G$.
    The non-zero eigenvalues of the outer product matrix $u \trans{v}$ are the same as the eigenvalues of $\trans{v} u$ \cite[Example~1.3.23]{Horn2013}
\end{proof}

If $H_{c}$ is limited to being a Toeplitz matrix, then the spectrum of $H_{d}$ can be bounded using a matrix symbol similar to the previous results for the primal Hessian.
\begin{theorem}
    \label{thm:dual:toep:eigBounds}
    Let $H_{c}$ in \eqref{eq:mpc:condMPC} be Toeplitz with the matrix symbol $P_{H_{c}} \in \Linf$, positive definite almost everywhere, then
    \begin{equation*}
        \lambda_{max}(H_{d}) \leq \lambda_{max}(P_{H_{d1}}) \leq \Hinfnorm{P_{H_{d1}}}
    \end{equation*}
    where
$
        P_{H_{d1}}(z) \coloneqq ( P_{H_{cQ}}(z))^{-1} \ctrans{P_{G}(z)} P_{G}(z)
$
with
$z \in \mathbb{T}. $
\end{theorem}
\begin{proof}
    The product $\trans{G} G$ is Toeplitz since $G$ is lower triangular.
    Results in \cite[Theorem~4.3]{Miranda2000} state that if $p, f \in \Linf$ are the matrix symbols for the Toeplitz matrices $P_{n}, F_{n}$ respectively, then the eigenvalues of $P_{n}^{-1} F_{n}$ lie inside $[\lambda_{min}(p^{-1}f), \lambda_{max}(p^{-1}f)]$.
    This result, combined with Lemma~\ref{lem:dual:toep:eigEquivalence} and $\lambda_{max} \leq \sigma_{max}$, gives the upper bound.
\end{proof}

\begin{figure}[tb]
    \centering
    
    \pgfplotsset{I_inst/.style={only marks, blue, every mark/.append style={solid, fill=white}, mark = *}}
    \pgfplotsset{I_limit_eig/.style={dashed, thick, blue}}
    \pgfplotsset{I_limit_svd/.style={densely dotted, thick, blue}}
    \pgfplotsset{IS_inst/.style={only marks, red, every mark/.append style={solid, fill=white}, mark = triangle}}
    \pgfplotsset{IS_limit_eig/.style={dashed, thick, red}}
    \pgfplotsset{IS_limit_svd/.style={densely dotted, thick, red}}
    \pgfplotsset{S_inst/.style={only marks, cyan, every mark/.append style={solid, fill=white}, mark = square}}
    \pgfplotsset{S_limit_eig/.style={dashed, thick, cyan}}
    \pgfplotsset{S_limit_svd/.style={densely dotted, thick, cyan}}
    
    \begin{tikzpicture}
        \pgfplotstableread[col sep=comma]{figures/data/spectrum_Dual_JonesMorari_Inst.csv}{\JMinstData}
        \pgfplotstableread[col sep=comma]{figures/data/spectrum_Dual_JonesMorari_Asymp.csv}{\JMasympData}
        \pgfplotstableread[col sep=comma]{figures/data/spectrum_Dual_MassSpring_Inst.csv}{\MSDinstData}
        \pgfplotstableread[col sep=comma]{figures/data/spectrum_Dual_MassSpring_Asymp.csv}{\MSDasympData}
           
        \begin{groupplot}[group style = {group name=JMplots, group size = 1 by 2, vertical sep=0.3cm}]

        \nextgroupplot[xmin   = 0,
                       xmax   = 30,
                       ymin   = 0.19,
                       ymax   = 0.196,
                       grid   = major,
                       ylabel = {$\lambda_{max}$ },
                       y tick label style={/pgf/number format/fixed,
                                           /pgf/number format/fixed zerofill,
                                           /pgf/number format/precision=3},
                       xticklabels=\empty,
                       height = 4cm,
                       width  = 8.5cm,
                       legend cell align=left,
                       legend columns=3,
                       legend to name=dualplotleg]
    
            \addplot[I_inst] table [x=Horizon, y=maxE_I] {\JMinstData};
            \addlegendentry{Input only$\quad$}
            
            \addplot[I_limit_eig, forget plot] table [x=Horizon, y=maxE_I_eig] {\JMasympData};

            \addplot[I_limit_svd, forget plot] table [x=Horizon, y=maxE_I_svd] {\JMasympData};
            
            \addlegendimage{S_inst}
            \addlegendentry{State only$\quad$}
            
            \addplot[IS_inst] table [x=Horizon, y=maxE_IS] {\JMinstData};
            \addlegendentry{Input and State}
            
            \addplot[IS_limit_eig, forget plot] table [x=Horizon, y=maxE_IS_eig] {\JMasympData};
            
            \addplot[IS_limit_svd, forget plot] table [x=Horizon, y=maxE_IS_svd] {\JMasympData};
            
            \addlegendimage{dashed, thick}
            \addlegendentry{$\lambda_{max}(P_{H_{d}})\quad$}

            \addlegendimage{densely dotted, thick}
            \addlegendentry{$\Hinfnorm{P_{H_{d}}}\quad$}
            
        \nextgroupplot[xmin   = 0,
                       xmax   = 30,
                       ymin   = 0.04,
                       ymax   = 0.09,
                       grid   = major,
                       ylabel = {$\lambda_{max}$ },
                       y tick label style={/pgf/number format/fixed,
                                           /pgf/number format/fixed zerofill,
                                           /pgf/number format/precision=3},
                       scaled y ticks = false,
                       xlabel = {$N$},
                       height = 3.5cm,
                       width  = 8.5cm]
        
            \addplot[S_inst] table [x=Horizon, y=maxE_S] {\JMinstData};
            
            \addplot[S_limit_eig] table [x=Horizon, y=maxE_S_eig] {\JMasympData};
            
            \addplot[S_limit_svd] table [x=Horizon, y=maxE_S_svd] {\JMasympData};
    \end{groupplot}
    
    \node at ($(JMplots c1r1.north west)!0.5!(JMplots c1r1.north east)+(-0.5cm,0.8cm)$) {\pgfplotslegendfromname{dualplotleg}};

    \end{tikzpicture}

    \caption{Maximum eigenvalue for the condensed dual Hessian matrix of System~\ref{sys:jonesMorari}. The lines represent the bounds computed using the results in Section~\ref{sec:specRes:dualHess}, and the markers represent the values of the dual matrix at that horizon.}
    \label{fig:specRes:dualMaxEig}
\end{figure}
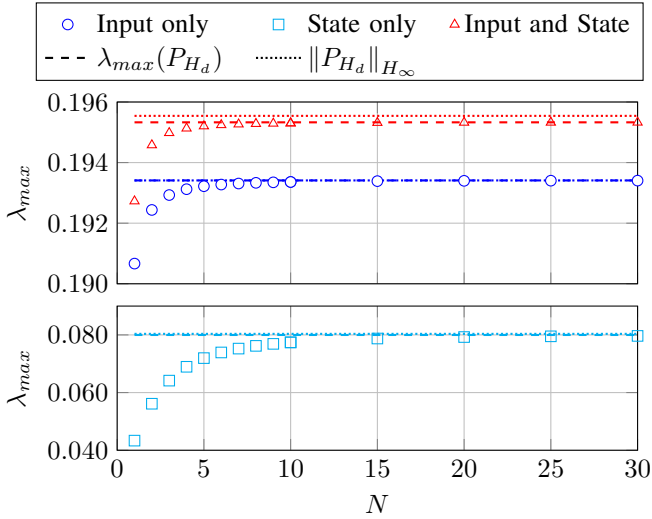

Using the results in Theorem~\ref{thm:dual:toep:eigBounds}, the spectrum of $H_d$ can be bounded using the matrix symbol of $H_{c}$ provided that $H_{c}$ is Toeplitz.
 Unfortunately, the computation of $\lambda_{max}$ for the symbol $P_{H_{d1}}$ requires an exhaustive search over the unit circle to find the largest eigenvalue since $H_{d1}$ is not symmetric.
 Instead, an upper bound on $\lambda_{max}$ can be found through the $H_{\infty}$ norm of $P_{H_{d1}}$, which is a faster operation.
 Figure~\ref{fig:specRes:dualMaxEig} shows $\Hinfnorm{P_{H_{d1}}}$ and the asymptotic properties of $\lambda_{max}$ for System~\ref{sys:jonesMorari} with either only input constraints, only state constraints, or both input and state constraints.
 Note that the bound $\lambda_{max}(H_{d}) \leq \lambda_{max}(P_{H_{d1}})$ in Theorem~\ref{thm:dual:toep:eigBounds} is tight, so there is a horizon above which equality occurs.
 Theorem~\ref{thm:dual:toep:eigBounds} also provides better bounds than the 2-norm estimate from \cite{Patrinos2013_DGP}, which estimates that $\lambda_{max}(P_{H_{d}})$ is less than $0.41$, $1.65$, and $2.06$ for input, state and both input and state constraints, respectively.

\section{The Effect of Weight Matrix Scaling}
\label{sec:weightScaling}

In this section, we examine how the scaling of the weighting matrices affects the extremal eigenvalues and condition number of the Hessian matrices from Section~\ref{sec:specRes}.
 These results provide an analytical grounding for the analysis conducted on the computational complexity in Section~\ref{sec:compComplexity}.

\subsection{Preliminaries}

The bounds on the eigenvalues and condition number we will develop in this section utilize the matrix trace normalized against the matrix size.
 To allow for horizon-independent bounds, we show that these trace normalizations can be computed in a horizon-independent manner through Lemma~\ref{lem:scaling:prelim:limitingValues}.
\begin{lemma}
    \label{lem:scaling:prelim:limitingValues}
    Let $H_{c}$ be the primal Hessian matrix of size $n \times n$ from Section~\ref{sec:specRes:condHess} with $P$ the solution to the discrete-time Lyapunov equation and $S$ an arbitrary matrix.
    Let the system $G_{s}$ be Schur-stable with $m$ inputs and the state space matrices $(A, B, I, 0)$. Define the following two dynamical systems
    \begin{align*}
     G_{Q} &\coloneqq \left\{
     \begin{array}{l}
        x^{+} = Ax + Bu\\
        y = \sqrtm{Q}x
     \end{array}
     \right.
     \\
     G_{QR} &\coloneqq \left\{
     \begin{array}{l}
        x^{+} = Ax + BR^{\sfrac{1}{2}}u\\
        y = \sqrtm{Q}x
     \end{array}
     \right.
    \end{align*}
    and let
    $
        I_{k} \coloneqq \frac{1}{2 \pi}\int_{0}^{2\pi} f_{k}( e^{j\omega}) d\omega
    $
    where
    \begin{equation*}
        \begin{array}{ll}
        f_{1} \coloneqq \trace{ \trans{S} P_{\Gamma} },
        & f_{2} \coloneqq \trace{R \trans{S} P_{\Gamma}},  \\
        f_{3} \coloneqq \trace{ \trans{S} P_{\Gamma} \trans{S} P_{\Gamma} },
        & f_{4} \coloneqq \trace{ \trans{S} P_{\Gamma} \ctrans{P_{\Gamma}} S }, \\
        f_{5} \coloneqq \trace{ \ctrans{P_{\Gamma}} Q P_{\Gamma} \trans{S} P_{\Gamma} }, 
        & f_{6} \coloneqq \trace{ \ctrans{P_{\Gamma}} Q P_{\Gamma} \ctrans{P_{\Gamma}} Q P_{\Gamma} }.
        \end{array}
    \end{equation*}
    Then the quantities
    $$ a \coloneqq \frac{ \trace{H_{c}} }{n}, \quad b \coloneqq \frac{ \trace{H_{c}^2} }{n}, $$
    have limits as $n \to \infty$ of
    \begin{align*}
        a_{l} &\coloneqq \underset{n \to \infty}{\lim} a = \frac{1}{m}\left( \Htwonorm{G_Q}^2 + 2I_{1} + \Fronorm{\sqrtm{R}}^2 \right), \\
        b_{l} &\coloneqq \underset{n \to \infty}{\lim} b 
        \!\begin{multlined}[t]
        = \frac{1}{m} \bigl( I_{6} + 4I_{5} + 4I_{2} + 2I_{3} + 2I_{4} \\
        + \Fronorm{R}^2 + 2\Htwonorm{G_{QR}}^2 \bigr).
        \end{multlined}
    \end{align*}
 \end{lemma}
 \begin{proof}
    See Appendix~\ref{app:proof:limitingValues}.
 \end{proof}
 Unlike the results in Section~\ref{sec:specRes:condHess} where multiple cases had to be examined, the result given in Lemma~\ref{lem:scaling:prelim:limitingValues} holds for every case presented in Section~\ref{sec:specRes:condHess} with no modifications.
 This occurs because the normalized trace of a finite rank matrix (such as $H_{e}$) goes to $0$ as $n \to \infty$, which leaves only the Toeplitz component of the Hessian.
 Results similar to Lemma~\ref{lem:scaling:prelim:limitingValues} could also be derived for other values of $P$, provided that $H_{c}$ can be decomposed into a Toeplitz component plus a finite-rank correction term.
 
To analyze the effect of scaling the weight matrices, we can utilize the linearity of the trace to scale the various terms in Lemma~\ref{lem:scaling:prelim:limitingValues}, giving the following result.
 \begin{lemma}
    \label{lem:scaling:prelim:normScaling}
    Let $\hat{Q} \coloneqq \alpha_1 Q$, $\hat{R} \coloneqq \alpha_2 R$ and $\hat{S} \coloneqq \alpha_3 S$ be scaled versions of the weight matrices. Then if the scaled matrices are substituted into the terms used in Lemma~\ref{lem:scaling:prelim:limitingValues}, the terms scale as
    \begin{gather*}
    \begin{align*}
        \hat{I}_{1} &\coloneqq \alpha_{3} I_{1}, &
        \hat{I}_{2} &\coloneqq \alpha_{2} \alpha_{3} I_{2}, &
        \hat{I}_{3} &\coloneqq \alpha_{3}^2 I_{3}, \\
        \hat{I}_{4} &\coloneqq \alpha_{3}^2 I_{4}, &
        \hat{I}_{5} &\coloneqq \alpha_{1} \alpha_{3} I_{5}, &
        \hat{I}_{6} &\coloneqq \alpha_{1}^2 I_{6}, \\
    \end{align*}\\
    \begin{align*}
         \Htwonorm{G_{\hat{Q}}}^2 &\coloneqq \alpha_1 \Htwonorm{G_{Q}}^2, & \norm{\hat{R}}_{F}^2 &\coloneqq \alpha_2^2 \Fronorm{R}^{2}, \\
         \Htwonorm{G_{\hat{Q}\hat{R}}}^2 &\coloneqq \alpha_1 \alpha_2 \Htwonorm{G_{QR}}^2, &  \Fronorm{\sqrtm{\hat{R}}}^2 &\coloneqq \alpha_2 \Fronorm{\sqrtm{R}}^{2}.
    \end{align*}
    \end{gather*}
 \end{lemma}
 \begin{proof}
    These results follow from the linearity of the trace and the integral operator.
 \end{proof}

\subsection{Extremal Eigenvalues}
\label{sec:weightScaling:eigenvalues}

In the complexity analysis of some algorithms, the extremal eigenvalues of the Hessian matrix (both dual and primal) appear as a factor.
 This means that in order to understand how the complexity scales with the weight matrices, it is important to understand how the extremal eigenvalues scale.

\subsubsection{Primal Hessian}

We begin by deriving bounds on the extremal eigenvalues for the primal Hessian matrix.
 \begin{lemma}
    \label{lem:weightScaling:eigen:bound}
    Let $H_{c}$ be the primal Hessian matrix from Section~\ref{sec:specRes:condHess}.
    Then 
    \begin{equation*}
        0 < \lambda_{min}(H_{c})\leq a_{l} \leq \lambda_{max}(H_{c})
    \end{equation*}
    where $a_l$ is defined in Lemma~\ref{lem:scaling:prelim:limitingValues}.
 \end{lemma}
 \begin{proof}
    Bounds on the extremal eigenvalues were given in \cite[Theorem~2.1]{Styan1983} as
    \begin{align*}
        m - sp &\leq \lambda_{min} \leq m - \sfrac{s}{p}\\
        m - \sfrac{s}{p} &\leq \lambda_{max} \leq m + sp
    \end{align*}
    where $m \coloneqq a$, $s \coloneqq \sqrt{b - a^2}$ and $p \coloneqq \sqrt{n-1}$.
    The limit of these bounds as $n \to \infty$ gives a finite upper bound for $\lambda_{min}$ and a finite lower bound for $\lambda_{max}$, both equal to $a_{l}$.
    The upper bound on $\lambda_{max}$ is in general not finite, and the lower bound for $\lambda_{min}$ is in general strictly greater-than $0$ since $H_{c}$ is positive definite.
 \end{proof}
 Equality can occur when $a^2 = b$ (e.g.\ $\trace{H_{c}} = \Fronorm{H_{c}}^2$), since that makes $s=0$.
 Since $H_{c}$ is positive definite though, this only can occur when all eigenvalues of $H_{c}$ are $1$.

The bound in Lemma~\ref{lem:weightScaling:eigen:bound} essentially creates a dividing line between the extremal eigenvalues.
 The effect of the weight matrix scaling on this dividing line can then be estimated.
 \begin{theorem}
    \label{thm:weightScaling:eigen:scalingBound}
    Let $\hat{H}_c$ be the primal Hessian matrix with the scaled weight matrices $\hat{Q} \coloneqq \alpha_1 Q$, $\hat{R} \coloneqq \alpha_2 R$ and $\hat{S} \coloneqq \alpha_3 S$.
    Then, the bound on the extremal eigenvalues given in Lemma~\ref{lem:weightScaling:eigen:bound} grows linearly with $\alpha_{1}$, $\alpha_{2}$ and $\alpha_{3}$.
 \end{theorem}
 \begin{proof}
    Combine the bounds in Lemma~\ref{lem:weightScaling:eigen:bound} with the scalings in Lemma~\ref{lem:scaling:prelim:normScaling}.
 \end{proof}

From Theorem~\ref{thm:weightScaling:eigen:scalingBound} it can be seen that the bounds on the extremal eigenvalues grow linearly with the scaling of the weight matrices.
 This is demonstrated in Figure~\ref{fig:scaledWeightMatrix:specMin} and~\ref{fig:scaledWeightMatrix:specMax} for the case when $S=0$, leaving only $\alpha_{1}$ and $\alpha_{2}$.
 When only one matrix is being scaled there are two distinct regions in the bound: $Q$ dominating and $R$ dominating.
 Since this bound is both an upper bound for $\lambda_{min}$ and a lower bound for $\lambda_{max}$, the transition between the two regions when $\alpha_1$ is being increased will occur earlier for $\lambda_{max}$ than $\lambda_{min}$.
 The region where $R$ dominates when only $Q$ is scaled is shown as a shaded region in Figure~\ref{fig:scaledWeightMatrix:specMin} and~\ref{fig:scaledWeightMatrix:specMax}.
 
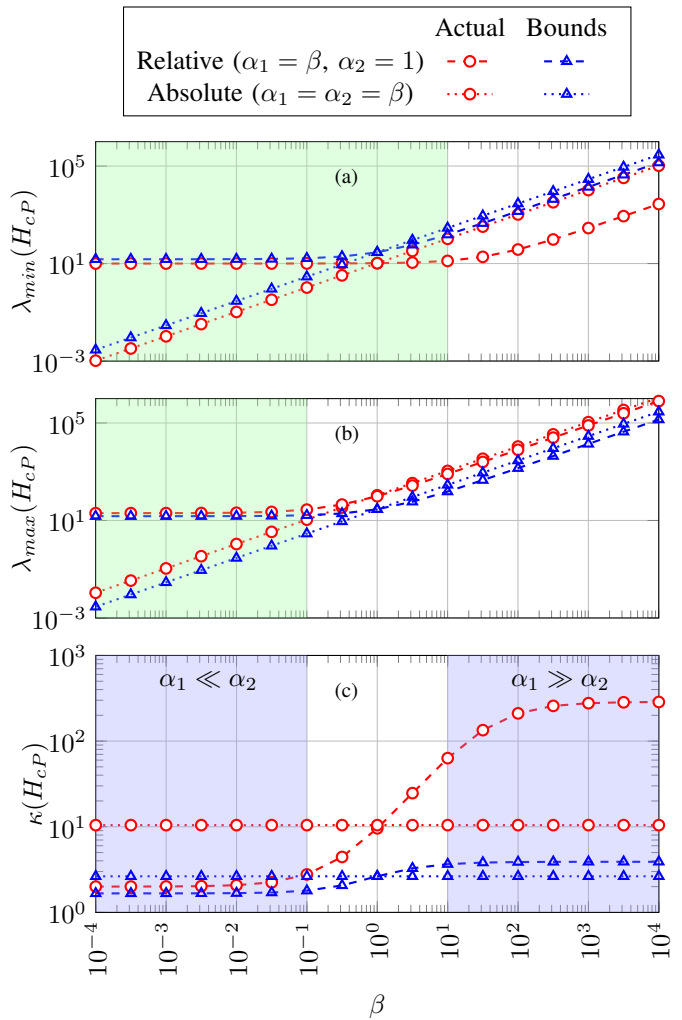
\begin{figure}[tb]
    \centering
    
    \pgfplotsset{abs_actual/.style={dotted, thick, red, every mark/.append style={solid, fill=white}, mark = *}}
    \pgfplotsset{rel_actual/.style={dashed, thick, red, every mark/.append style={solid, fill=white}, mark = *}}
    \pgfplotsset{abs_bound/.style={dotted, thick, blue, every mark/.append style={solid, fill=gray}, mark = triangle}}
    \pgfplotsset{rel_bound/.style={dashed, thick, blue, every mark/.append style={solid, fill=gray}, mark = triangle}}
    
    \pgfplotsset{
        legend image with text/.style={
            legend image code/.code={%
                \node[anchor=center] at (0.3cm,0cm) {#1};
            }
        },
    }
    
    \begin{tikzpicture}
        \pgfplotstableread[col sep=comma]{figures/data/WeightScaling_Absolute_JonesMorari.csv}{\absdata}
        \pgfplotstableread[col sep=comma]{figures/data/WeightScaling_Relative_JonesMorari.csv}{\reldata}

        \begin{groupplot}[group style = {group name=plots, group size = 1 by 3, horizontal sep=0cm, vertical sep=0.5cm}]
            
            \nextgroupplot[xmin   = 0.0001,
                           xmax   = 10000,
                           ymin   = 0.001,
                           ymax   = 1000000,
                           xmajorticks = false,
                           grid = major,
                           ylabel = {$\lambda_{min}(H_{cP})$},
                           ylabel shift = -9 pt,
                           height = 4.5cm,
                           width  = 0.5\textwidth,
                           ymode  = log,
                           xmode  = log,
                           title = {\pgfplotslegendfromname{condPlotleg}}]
                 
                \addplot[abs_actual] table [x=Scaling, y=minE] {\absdata};
                
                \addplot[rel_actual] table [x=Scaling, y=minE] {\reldata};
                
                \addplot[rel_bound] table [x=Scaling, y=minE_ub] {\reldata};

                \addplot[abs_bound] table [x=Scaling, y=minE_ub] {\absdata};

                \addplot[ draw=none, fill=green!40, opacity=0.3] coordinates {(10, 1000000) (0.0001, 1000000) (0.0001, 0.001) (10, 0.001)} \closedcycle;

            \nextgroupplot[xmin   = 0.0001,
                           xmax   = 10000,
                           ymin   = 0.001,
                           ymax   = 1000000,
                           xmajorticks = false,
                           grid = major,
                           ylabel = {$\lambda_{max}(H_{cP})$},
                           ylabel shift = -9 pt,
                           height = 4.5cm,
                           width  = 0.5\textwidth,
                           ymode  = log,
                           xmode  = log]

                \addplot[abs_actual] table [x=Scaling, y=maxE] {\absdata};
                
                \addplot[rel_actual] table [x=Scaling, y=maxE] {\reldata};
                
                \addplot[rel_bound] table [x=Scaling, y=maxE_lb] {\reldata};
                
                \addplot[abs_bound] table [x=Scaling, y=maxE_lb] {\absdata};
                
                \addplot[ draw=none, fill=green!40, opacity=0.3] coordinates {(0.1, 1000000) (0.0001, 1000000) (0.0001, 0.001) (0.1, 0.001)} \closedcycle;

            \nextgroupplot[xmin   = 0.0001,
                           xmax   = 10000,
                           ymin   = 1,
                           ymax   = 1000,
                           grid = major,
                           xlabel = {$\beta$},
                           ylabel = {$\kappa(H_{cP})$},
                           ylabel shift = -7 pt,
                           x tick label style={rotate=90,anchor=east},
                           height = 5cm,
                           width  = 0.5\textwidth,
                           ymode  = log,
                           xmode  = log,
                           legend columns=3,
                           legend to name=condPlotleg]

                \addlegendimage{empty legend}
                \addlegendentry{}
                
                \addlegendimage{legend image with text=Actual}
                \addlegendentry{}

                \addlegendimage{legend image with text=Bounds}
                \addlegendentry{}
                
                \addlegendimage{empty legend}
                \addlegendentry{Relative ($\alpha_1 = \beta$, $\alpha_2 = 1$)}
                
                \addplot[rel_actual] table [x=Scaling, y=K] {\reldata};
                \addlegendentry{}
                
                \addplot[rel_bound] table [x=Scaling, y=K_lb] {\reldata};
                \addlegendentry{}

                \addlegendimage{empty legend}
                \addlegendentry{Absolute ($\alpha_1 = \alpha_2 = \beta$)}
                
                \addplot[abs_actual] table [x=Scaling, y=K] {\absdata};
                \addlegendentry{}
                
                \addplot[abs_bound] table [x=Scaling, y=K_lb] {\absdata};
                \addlegendentry{}
                
                \addplot[ draw=none, fill=blue!40, opacity=0.3] coordinates {(0.0001, 1) (0.0001, 1000) (0.1, 1000) (0.1, 1)} \closedcycle;
                \addplot[ draw=none, fill=blue!40, opacity=0.3] coordinates {(10000, 1) (10000, 1000) (10, 1000) (10, 1)} \closedcycle;
                \node[anchor=center] at (0.004,500) {$\alpha_1 \ll \alpha_2$};
                \node[anchor=center] at (400,500) {$\alpha_1 \gg \alpha_2$};
                
        \end{groupplot}
        
        \node [text width=1em,anchor=center] at ($(plots c1r1.north)-(1.1em,0.7em)$) {\subfloat[]{\label{fig:scaledWeightMatrix:specMin}}};
        \node [text width=1em,anchor=center] at ($(plots c1r2.north)-(1.1em,0.7em)$) {\subfloat[]{\label{fig:scaledWeightMatrix:specMax}}};
        \node [text width=1em,anchor=center] at ($(plots c1r3.north)-(1.1em,0.7em)$) {\subfloat[]{\label{fig:scaledWeightMatrix:cond}}};
    \end{tikzpicture}
    
    \caption{The bounds on the extremal eigenvalues from Lemma~\ref{lem:weightScaling:eigen:bound} and the condition number from Theorem~\ref{thm:condScaling:scaledBound} are compared against the actual asymptotic results from Theorem~\ref{thm:dense:PLyap:hessEig} for various scalings of $Q$ and $R$ matrices in System~\ref{sys:jonesMorari}.}
    \label{fig:scaledWeightMatrix}
\end{figure}

\subsubsection{Dual Hessian}

For the dual Hessian matrix, the largest eigenvalue can be bounded through the spectral norm, as was done in Proposition~\ref{prop:dense:dual:arbMatrixUB}.
 This bound is affected by both the spectrum of the constraint matrix $G$ and the primal Hessian matrix $H_{c}$.
 When the weight matrices are scaled, only the primal Hessian is affected, which means only the $\lambda_{min}(H_c)$ term in the denominator will change.
 Unfortunately, since the lower bound in Theorem~\ref{thm:weightScaling:eigen:scalingBound} is 0, a direct upper bound on $\norm{H_d}_2$ cannot be computed.

We can instead use the upper bound for $\lambda_{min}$ given in Lemma~\ref{lem:weightScaling:eigen:bound} to examine how the bound from Proposition~\ref{prop:dense:dual:arbMatrixUB} changes with weight scaling.
 It is known from Theorem~\ref{thm:weightScaling:eigen:scalingBound} that there is a linear relation between the magnitude of the cost matrices and the upper bound for $\lambda_{min}$.
 This implies that there will be an inverse relation between the cost scaling and $\lambda_{max}$ of the dual Hessian:
 as either $\alpha_1$, $\alpha_2$ or $\alpha_3$ grow, the bound on $\hat{\gamma}$ will shrink.


\subsection{Condition Number}
\label{sec:weightScaling:cond}

We focus the analysis in this section on the primal Hessian~$H_{p}$, since the condition number of the dual Hessian~$H_{d}$ is in general unbounded due to $H_{d}$ being positive semidefinite.
 The results presented in Section~\ref{sec:specRes:condHess} provide a means of calculating the condition number estimates for a given set of weighting matrices, but provide little intuition into the general effect of matrix scaling.
 To examine the effect of scaling, we utilize a lower bound for the condition number of the Hessian matrix.
\begin{lemma}
    \label{lem:condScaling:condLowerBound}
    Let $a_{l}$ and $b_{l}$ be defined in Lemma~\ref{lem:scaling:prelim:limitingValues}.
    The primal Hessian matrix $H_c$ has a lower bound on the condition number given by
    \begin{equation*}
       \kappa(H_c) \geq 1 + 2 \frac{\sqrt{ b_{l} - a_{l}^2}}{a_{l}}
    \end{equation*}
\end{lemma}
\begin{proof}
    See Appendix \ref{app:condAsympBoundProof}.
\end{proof}

The lower bound presented in Lemma~\ref{lem:condScaling:condLowerBound} is horizon-independent, and holds for any choice of $S$ and either $P = Q$ or $P$ the solution to the discrete-time Lyapunov equation.
 This lower bound represents the best possible condition number that can be obtained.
 Unfortunately, knowledge of the worst possible condition number (e.g.\ an upper bound) cannot be obtained in a horizon-independent manner since lower-bounds on the smallest eigenvalue of $H_c$ go to $0$ as the matrix size increases.
 
To more closely examine the effect of the matrix scaling, we examine the case when $S = 0$ and only $Q$ and $R$ are scaled.
 We assume that $P$ is either $Q$ or the solution to the discrete-time Lyapunov equation.
 For this case, the lower bound from Lemma~\ref{lem:condScaling:condLowerBound} becomes the bound given in Theorem~\ref{thm:condScaling:scaledBound}.
\begin{theorem}
    \label{thm:condScaling:scaledBound}
    Let $\hat{H}_c$ be the Hessian matrix with the scaled weight matrices $\hat{Q} \coloneqq \alpha_1 Q$, $\hat{R} \coloneqq \alpha_2 R$ and $S = 0$. Then given the dynamical systems and the integrals in Lemma~\ref{lem:scaling:prelim:limitingValues}, a lower bound for the condition number of $\hat{H}_c$ is
    \begin{equation*}
        \kappa(\hat{H}_c) \geq 1 + 2 \frac{\sqrt{ \alpha_1^2 n_1 + 2\alpha_1 \alpha_2 n_2 + \alpha_2^2 n_3}}{\alpha_1 \Htwonorm{G_{Q}}^2 + \alpha_2 \Fronorm{\sqrtm{R}}^{2}},
    \end{equation*}
    with
    \begin{equation*}
        \begin{array}{ll}
            n_1 \coloneqq m I_{6} - \Htwonorm{G_{Q}}^4, & n_3 \coloneqq m \Fronorm{R}^{2} - \Fronorm{\sqrtm{R}}^{4}, \\
            \multicolumn{2}{l}{ n_2 \coloneqq m \Htwonorm{G_{QR}}^2 - \Htwonorm{G_{Q}}^2 \Fronorm{R^{\sfrac{1}{2}}}^{2}. }
        \end{array}
    \end{equation*}
\end{theorem}
\begin{proof}
    See Appendix~\ref{app:scaledBoundProof}.
\end{proof}

A numerical example for the results in Theorem~\ref{thm:condScaling:scaledBound} is presented for System~\ref{sys:jonesMorari} in Figure~\ref{fig:scaledWeightMatrix}.
 Examining the behaviour of the bound, it can be seen that there exist three distinct regions (shown through shading in Figure~\ref{fig:scaledWeightMatrix:cond}): when $\alpha_1 \ll \alpha_2$, when $\alpha_1 \gg \alpha_2$, and the transition region.
 The lower bounds for the regions when $\alpha_1 \ll \alpha_2$ and $\alpha_1 \gg \alpha_2$ can be estimated through the following corollary.
\begin{corollary}
    \label{corr:weightScaling:asympBounds}
    The lower bound in Theorem~\ref{thm:condScaling:scaledBound} has asymptotic values
    \begin{equation*}
        \kappa(\hat{H}_c) \geq 
        \begin{cases}
        1 + 2 \frac{\sqrt{m I_{6} - \Htwonorm{G_{Q}}^4}}{\Htwonorm{G_{Q}}^2} & \text{if $\alpha_1 \gg \alpha_2$}\\ 
        1 + 2 \frac{\sqrt{m\Fronorm{R}^{2} - \Fronorm{\sqrtm{R}}^{4}}}{\Fronorm{\sqrtm{R}}^{2}} & \text{if $\alpha_1 \ll \alpha_2$} 
        \end{cases}
    \end{equation*}
\end{corollary}

When $\alpha_1 \ll \alpha_2$, the spectrum of $R$ dominates the condition number.
 The transition region is caused by the fact that $\lambda_{max}$ begins growing before $\lambda_{min}$ when $\alpha_1$ is scaled, causing their ratio to change.
 Then when $\lambda_{min}$ also begins growing, the ratio becomes constant leading to the region where $\alpha_1 \gg \alpha_2$.
 In this region, the condition number is dominated by the singular value distribution of the dynamical system with an output mapping through the weighting matrix $Q$.

The bounds in these regions are related to the spread and mean of the spectrum of the matrices/system.
 The quantity in the numerators of Corollary~\ref{corr:weightScaling:asympBounds} can be viewed as an upper bound on the spread of the spectrum (the largest distance between two eigenvalues) \cite{Mirsky1956_matrixSpread}, while the denominator can be viewed as the mean of the spectrum.
 We can use this relation to see that for large ratios of $R$ to $Q$, $\kappa$ is dominated by the spread of the spectrum of $R$ over its average.
 Alternatively, for large $Q$ to $R$ ratios the bound is dominated by the spread of the singular values of the physical system with an output compensator of $\sqrtm{Q}$ over the average of the singular values.
 
Another interesting phenomenon arises when both $Q$ and $R$ are scaled by the same amount.
 \begin{corollary}
    \label{corr:weightScaling:constantValue}
    If the relative scaling of the two weight matrices is held constant at $\alpha_1 = \eta \alpha_2$ for a constant $\eta > 0$, then the lower bound is constant with the value
    \begin{equation*}
        \kappa(\hat{H}_c) \geq 1 + 2 \frac{\sqrt{\eta^2 n_1 + 2\eta n_2 + n_3}}{ \eta \Htwonorm{G_{Q}}^2 + \Fronorm{\sqrtm{R}}^{2}}.
    \end{equation*}
 \end{corollary}
 
 Essentially, if both $Q$ and $R$ are scaled equally (e.g.\ $\alpha_1 = \alpha_2$), then the condition number of the Hessian matrix does not change.
 This is also true when $S \neq 0$ and is scaled equally with $Q$ and $R$ (e.g.\ $\alpha_1 = \alpha_2 = \alpha_3$), so it is only when the matrices are scaled separately that the condition number changes.

\subsection{Condition Number with Discretized Weights}
\label{sec:weightScaling:condDiscreteWeights}

The results inside Section~\ref{sec:weightScaling:cond} examine the effect of the scaling of the weights in the discrete-time problem~\eqref{eq:mpc:linMPC}.
 Alternatively, the weighting matrices can be generated from the continuous-time problem through a discretization procedure given in \cite{Bini2014_OptimalSamplingPattern}, where $\tau$ is the sampling time and
 \begin{align}
     \Phi(\tau) &\coloneqq e^{A\tau}, \quad \Gamma(\tau) \coloneqq \int_{0}^{\tau} e^{A(\tau - t)}dt B, \nonumber\\
     Q_{d}(\tau) &\coloneqq \int_{0}^{\tau} \ctrans{\Phi}(t) Q_c \Phi(t) dt, \label{eq:discWeights:Q}\\
     S_{d}(\tau) &\coloneqq \int_{0}^{\tau} \ctrans{\Phi}(t) Q_{c} \Gamma(t) dt, \label{eq:discWeights:S}\\
     R_{d}(\tau) &\coloneqq \tau R_{c} + \int_{0}^{\tau} \ctrans{\Gamma}(t) Q_{c} \Gamma(t) dt. \label{eq:discWeights:R}
 \end{align}
 Computation of the weights in this way will cause the discrete-time cost to be equivalent to the continuous-time cost.
 
Scaling the weighting matrices in the continuous-time problem by $\hat{Q_{c}} \coloneqq \alpha_{1} Q_{c}$ and $\hat{R_{c}} \coloneqq \alpha_{2} R_{c}$ then leads to the following scalings for the discrete-time matrices
 \begin{align*}
    \hat{Q_{d}} = \alpha_{1} Q_{d}, \quad
    \hat{S_{d}} = \alpha_{1} S_{d}, \quad 
    \hat{R_{d}} = \alpha_{1} R_{d} + (\alpha_{2} - \alpha_{1}) \tau R_{c}
 \end{align*}
 Note that scaling $\alpha_{1}$ affects all three weighting matrices, but $\alpha_2$ only affects $R_{d}$.

This scaling behaves very similarly to the scaling of the discrete-time matrices in Section~\ref{sec:weightScaling:cond}.
 When $\alpha_1 = \alpha_2$, all matrices are scaled equally and the results from Corollary~\ref{corr:weightScaling:constantValue} say that the condition number will not change.

\section{Computational Complexity}
\label{sec:compComplexity}

In this section, we examine the change in computational complexity for the DGP and FGM algorithms when applied to the CLQR problem with scaled weight matrices.
 We specifically focus on the relative scaling case, which means that $R$ is held constant and $Q$ is scaled by $\alpha_1 \in [ 10^{-4}, 10^{6}]$.

\subsection{Fast Gradient Method}

 
When the upper bounds for $\Delta$ and $\epsilon$ in~\eqref{eq:fgm:constBounds} are used in the UIB for the Fast Gradient Method, the UIB becomes dependent only on $\kappa$.
 Further simplification shows that $a \geq b > 0$ for $\kappa \geq 1$, which leads to
 \begin{equation}
    \label{eq:compComplex:fgm:simpleUIB}
    \text{UIB} = \ceil*{2\sqrt{\kappa} - 2}.
 \end{equation}
 The square-root dependence of~\eqref{eq:compComplex:fgm:simpleUIB} on $\kappa$ means that the UIB will follow the same general trend as the condition number.
 Corollary~\ref{corr:weightScaling:asympBounds} can then be used to investigate when the UIB for FGM will be small.
 For instance, Corollary~\ref{corr:weightScaling:asympBounds} suggests that $\kappa$ will be smaller if $\alpha_2 \gg \alpha_1$ (e.g.\ the inputs are more heavily weighted than the states) for System~\ref{sys:jonesMorari} .
 This will then lead to a smaller iteration bound for FGM when the inputs are more heavily weighted, which is seen in practice (see Figure~\ref{fig:alg:fgm:iter}).

\begin{figure}[tb]
    \centering
    \begin{tikzpicture}
    \begin{groupplot}[group style = {group name=plots, group size = 1 by 4, horizontal sep=1.5cm, vertical sep=0.5cm}]
        \pgfplotstableread[col sep=comma]{figures/data/complexity_FGM_JonesMorari.csv}{\JMdata}

%
       
         \nextgroupplot[xmin   = 0.0001,
                        xmax   = 10000,
                        ymin   = 1,
                        ymax   = 100,
                        xmajorticks = false,
                        xtick distance = 10,
                        grid = major,
                        ylabel = {UIB},
                        height = 3cm,
                        width  = 0.48\textwidth,
                        ymode  = log,
                        xmode  = log,
                        legend pos = north west,
                        legend cell align=left]
    
            \addplot[dotted, thick, blue, every mark/.append style={solid, fill=white}, mark = *] table [x=Scaling, y=UIB] {\JMdata};
            
        \nextgroupplot[xmin   = 0.0001,
                       xmax   = 10000,
                       ymin   = -30,
                       ymax   = 250,
                       xmajorticks = false,
                       xtick distance = 10,
                       grid = major,
                       ylabel = {\% Difference},
                       height = 4cm,
                       width  = 0.48\textwidth,
                       xmode  = log,
                       legend style={at={(0.97,0.4)},anchor=east},
                       legend cell align=left]
            
            \addplot[dashed, thick, red, every mark/.append style={solid, fill=white}, mark = triangle] table [x=Scaling, y=pdiffInput] {\JMdata};
            \addlegendentry{Input}
            
            \addplot[dashed, thick, every mark/.append style={solid, fill=white}, mark = square] table [x=Scaling, y=pdiffIter] {\JMdata};
            \addlegendentry{UIB}
            
        \nextgroupplot[xmin   = 0.0001,
                       xmax   = 10000,
                       ymin   = -20,
                       ymax   = 30,
                       xmajorticks = true,
                       xtick distance = 10,
                       grid = major,
                       xlabel = {$\alpha_1$},
                       ylabel = {\% Difference},
                       height = 4cm,
                       width  = 0.48\textwidth,
                       xmode  = log,
                       legend style={at={(0.025,0.79)},anchor=west},
                       legend cell align=left]
            
            \addplot[dashed, thick, red, every mark/.append style={solid, fill=white}, mark = triangle] table [x=Scaling, y=pdiffState] {\JMdata};
            \addlegendentry{State}
            
            \addplot[dashed, thick, every mark/.append style={solid, fill=white}, mark = square] table [x=Scaling, y=pdiffSettling] {\JMdata};
            \addlegendentry{Settling Time}

    \end{groupplot}

    \node [text width=1em,anchor=center] at ($(plots c1r1.north)-(0.5em,0.0em)$) {\subfloat[]{\label{fig:alg:fgm:iter}}};
    \node [text width=1em,anchor=center] at ($(plots c1r2.north)-(0.5em,0.0em)$) {\subfloat[]{\label{fig:alg:fgm:pdiff1}}};
    \node [text width=1em,anchor=center] at ($(plots c1r3.north)-(0.5em,0.0em)$) {\subfloat[]{\label{fig:alg:fgm:pdiff2}}};
    \end{tikzpicture}
    
    \caption{Effect of scaling the weight matrix $Q$ by $\alpha_1$ and holding $R$ constant on the Fast Gradient Method when solving System~\ref{sys:jonesMorari} with no state constraints.}
    \label{fig:alg:fgm}

\end{figure}
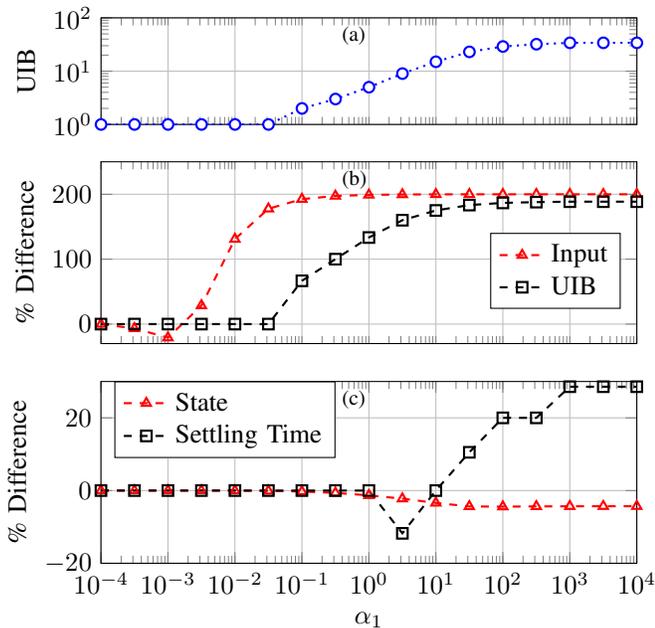

The performance of the CLQR controller using FGM was examined by controlling System~\ref{sys:jonesMorari} without the state constraints and with a chosen suboptimality level of $\delta_{max} = 0.001$ and $N=20$.
 The CLQR regulates System~\ref{sys:jonesMorari} starting from the initial condition of 
 $x_{0} = \trans{[ 0.1, 0.1, 0.1, 0.1 ]}$
 to the origin.
 The performance was measured by taking the 2-norm of the state and input trajectories.
 To compare the performance at each scaling factor, the percent difference versus $\alpha_{1} = 10^{-4}$ was computed, and is shown in Figures~\ref{fig:alg:fgm:pdiff1} and~\ref{fig:alg:fgm:pdiff2}.
  
For this problem, the overall change in the 2-norm of the state trajectories across the entire scaling range was less than 5\%, while the input norm varied by approximately 200\%.
 Additionally, the number of iterations required varied by 188\% across the scaling range.
 An interesting feature of this is that the change in the norm of the input occurs at a different time than the change in the UIB.
 These results show that if the performance of the system states was the design criteria, an aggressive weighting will only produce a 5\% decrease in the norm of the state trajectories, but will produce a 188\% increase in the number of iterations required for the solver.

\subsection{Dual Gradient Projection}

\begin{figure}[tb]
    \def\Nmin{0.0001}
    \def\Nmax{1000000}
    \centering
    
    \begin{tikzpicture}
    \begin{groupplot}[group style = {group name=plots, group size = 1 by 5, horizontal sep=1.5cm, vertical sep=0.5cm}]
        \pgfplotstableread[col sep=comma]{figures/data/complexity_DGP_JonesMorari.csv}{\JMdata}

         \nextgroupplot[xmin   = \Nmin,
                        xmax   = \Nmax,
                        ymin   = 1e-6,
                        ymax   = 10,
                        xmajorticks = false,
                        xtick distance = 10,
                        grid = major,
                        ylabel = {$\lambda_{max}(H_{d})$},
                        height = 3.75cm,
                        width  = 0.48\textwidth,
                        ymode  = log,
                        xmode  = log,
                        legend pos=south west]
        
            \addplot[dashed, thick, red, every mark/.append style={solid, fill=white}, mark = triangle] table [x=Scaling, y=asymp_maxE_d] {\JMdata};
            \addlegendentry{Theorem~\ref{thm:dual:toep:eigBounds}}
            
            \addplot[dotted, thick, blue, every mark/.append style={solid, fill=white}, mark = square] table [x=Scaling, y=actual_maxE_d] {\JMdata};
            \addlegendentry{Actual}
            
            \addplot[dotted, thick, black, every mark/.append style={solid, fill=white}, mark = diamond] table [x=Scaling, y=subMult_maxE_d] {\JMdata};
            \addlegendentry{\cite{Patrinos2013_DGP}}

         \nextgroupplot[xmin   = \Nmin,
                        xmax   = \Nmax,
                        ymin   = 50,
                        ymax   = 300,
                        xmajorticks = false,
                        xtick distance = 10,
                        grid = major,
                        ylabel = {UDB},
                        height = 3cm,
                        width  = 0.48\textwidth,
                        xmode  = log]
         
             \addplot[dotted, thick, blue, every mark/.append style={solid, fill=white}, mark = *] table [x=Scaling, y=UDB] {\JMdata};
         
         \nextgroupplot[xmin   = \Nmin,
                        xmax   = \Nmax,
                        ymin   = 1000,
                        ymax   = 1e10,
                        xmajorticks = false,
                        xtick distance = 10,
                        grid = major,
                        ylabel = {UIB},
                        height = 3.75cm,
                        width  = 0.48\textwidth,
                        ymode  = log,
                        xmode  = log,
                        legend style={at={(0.01,0.35)},anchor=west},
                        legend cell align=left]
    
            \addplot[dotted, thick, blue, every mark/.append style={solid, fill=white}, mark = *] table [x=Scaling, y=UIB] {\JMdata};
            \addlegendentry{Theorem~\ref{thm:dual:toep:eigBounds}}

            \addplot[dashed, thick, black, every mark/.append style={solid, fill=white}, mark = diamond] table [x=Scaling, y=UIB_subMult] {\JMdata};
            \addlegendentry{$\lambda_{max}$ from \cite{Patrinos2013_DGP}}
            
            
        \nextgroupplot[xmin   = \Nmin,
                       xmax   = \Nmax,
                       ymin   = -200,
                       ymax   = 250,
                       xmajorticks = false,
                       xtick distance = 10,
                       grid = major,
                       ylabel = {\% Difference},
                       height = 3.75cm,
                       width  = 0.48\textwidth,
                       xmode  = log,
                       legend style={at={(1.0,0.60)},anchor=east},
                       legend cell align=left]
            
            \addplot[dashed, thick, red, every mark/.append style={solid, fill=white}, mark = triangle] table [x=Scaling, y=pdiffInput] {\JMdata};
            \addlegendentry{Input}
            
            \addplot[dashed, thick, every mark/.append style={solid, fill=white}, mark = square] table [x=Scaling, y=pdiffIter] {\JMdata};
            \addlegendentry{UIB}
            
        \nextgroupplot[xmin   = \Nmin,
                       xmax   = \Nmax,
                       ymin   = -20,
                       ymax   = 30,
                       xmajorticks = true,
                       xtick distance = 10,
                       grid = major,
                       xlabel = {$\alpha_1$},
                       ylabel = {\% Difference},
                       height = 4cm,
                       width  = 0.48\textwidth,
                       xmode  = log,
                       xticklabel style={rotate=270},
                       legend style={at={(0.025,0.79)},anchor=west},
                       legend cell align=left]
            
            \addplot[dashed, thick, red, every mark/.append style={solid, fill=white}, mark = triangle] table [x=Scaling, y=pdiffState] {\JMdata};
            \addlegendentry{State}
            
            \addplot[dashed, thick, every mark/.append style={solid, fill=white}, mark = square] table [x=Scaling, y=pdiffSettling] {\JMdata};
            \addlegendentry{Settling Time}

    \end{groupplot}

    \node [text width=1em,anchor=center] at ($(plots c1r1.north)-(0.5em,0.6em)$) {\subfloat[]{\label{fig:alg:dgp:sys1:eps}}};
    \node [text width=1em,anchor=center] at ($(plots c1r2.north)-(0.5em,0.0em)$) {\subfloat[]{\label{fig:alg:dgp:sys1:udb}}};
    \node [text width=1em,anchor=center] at ($(plots c1r3.north)-(0.5em,0.0em)$) {\subfloat[]{\label{fig:alg:dgp:sys1:iter}}};
    \node [text width=1em,anchor=center] at ($(plots c1r4.north)-(0.5em,0.6em)$) {\subfloat[]{\label{fig:alg:dgp:sys1:pdiff1}}};
    \node [text width=1em,anchor=center] at ($(plots c1r5.north)-(0.5em,0.0em)$) {\subfloat[]{\label{fig:alg:dgp:sys1:pdiff2}}};
    \end{tikzpicture}
    
    \caption{Effect of scaling the weight matrix $Q$ by $\alpha_1$ and holding $R$ constant on the Dual Gradient Projection method when solving System~\ref{sys:jonesMorari} with no state constraints.}
    \label{fig:alg:dgp:sys1}

\end{figure}
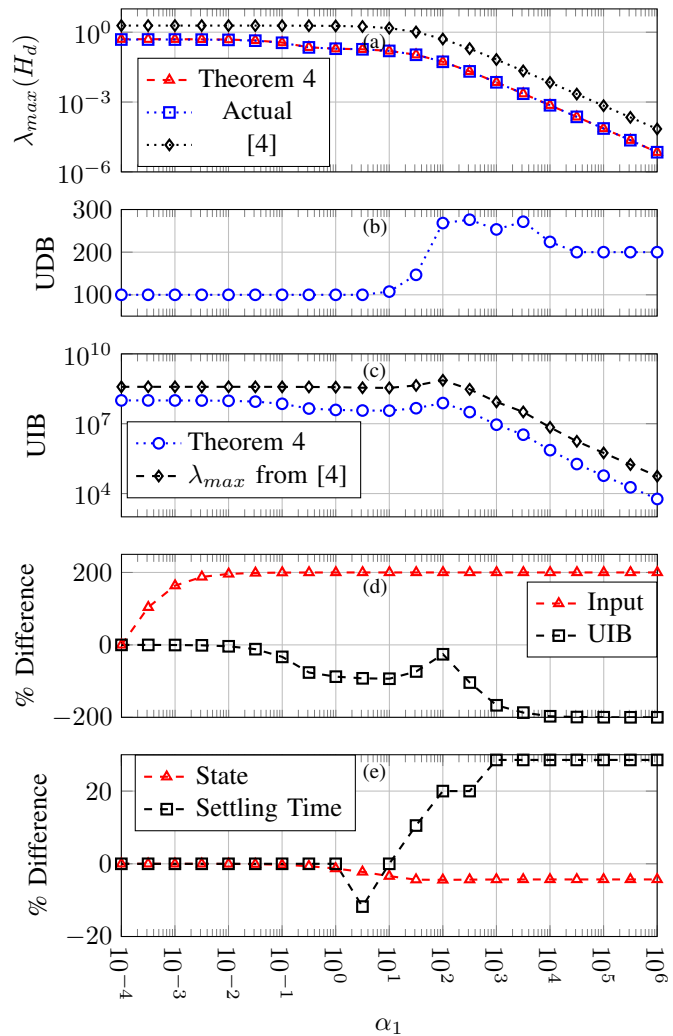

\begin{figure}[tb]
    \centering
    
    \def\Nmin{0.0001}
    \def\Nmax{1000000}
    
    \begin{tikzpicture}
    \begin{groupplot}[group style = {group name=plots, group size = 1 by 4, horizontal sep=1.5cm, vertical sep=0.5cm}]
        \pgfplotstableread[col sep=comma]{figures/data/complexity_DGP_MassSpring.csv}{\MSdata}

         \nextgroupplot[xmin   = \Nmin,
                        xmax   = \Nmax,
                        ymin   = 1e-5,
                        ymax   = 1e10,
                        xmajorticks = false,
                        xtick distance = 10,
                        grid = major,
                        ylabel = {$\lambda_{max}(H_{d})$},
                        height = 3.25cm,
                        width  = 0.48\textwidth,
                        ymode  = log,
                        xmode  = log]
        
            \addplot[dashed, thick, red, every mark/.append style={solid, fill=white}, mark = triangle] table [x=Scaling, y=asymp_maxE_d] {\MSdata};
            \addlegendentry{Proposition~\ref{prop:dense:dual:arbMatrixUB}}

            \addplot[dotted, thick, blue, every mark/.append style={solid, fill=white}, mark = square] table [x=Scaling, y=actual_maxE_d] {\MSdata};
            \addlegendentry{Actual}
            
         \nextgroupplot[xmin   = \Nmin,
                        xmax   = \Nmax,
                        ymin   = 10000000,
                        ymax   = 1e14,
                        xmajorticks = false,
                        xtick distance = 10,
                        grid = major,
                        ylabel = {UIB},
                        height = 3cm,
                        width  = 0.48\textwidth,
                        ymode  = log,
                        xmode  = log,
                        legend pos = north west,
                        legend cell align=left]
    
            \addplot[dotted, thick, blue, every mark/.append style={solid, fill=white}, mark = *] table [x=Scaling, y=UIB] {\MSdata};


        \nextgroupplot[xmin   = \Nmin,
                       xmax   = \Nmax,
                       ymin   = -250,
                       ymax   = 250,
                       xmajorticks = false,
                       xtick distance = 10,
                       grid = major,
                       ylabel = {\% Difference},
                       height = 4cm,
                       width  = 0.48\textwidth,
                       xmode  = log,
                       legend style={at={(0.97,0.5)},anchor=east},
                       legend cell align=left]
            
            \addplot[dashed, thick, red, every mark/.append style={solid, fill=white}, mark = triangle] table [x=Scaling, y=pdiffInput] {\MSdata};
            \addlegendentry{Input}
            
            \addplot[dashed, thick, every mark/.append style={solid, fill=white}, mark = square] table [x=Scaling, y=pdiffIter] {\MSdata};
            \addlegendentry{UIB}
            
        \nextgroupplot[xmin   = \Nmin,
                       xmax   = \Nmax,
                       ymin   = -200,
                       ymax   = 0,
                       xmajorticks = true,
                       xtick distance = 10,
                       grid = major,
                       xlabel = {$\alpha_1$},
                       ylabel = {\% Difference},
                       height = 4cm,
                       width  = 0.48\textwidth,
                       xmode  = log,
                       xticklabel style={rotate=270},
                       legend style={at={(1.0,0.79)},anchor=east},
                       legend cell align=left]
            
            \addplot[dashed, thick, red, every mark/.append style={solid, fill=white}, mark = triangle] table [x=Scaling, y=pdiffState] {\MSdata};
            \addlegendentry{State}
            
            \addplot[dashed, thick, every mark/.append style={solid, fill=white}, mark = square] table [x=Scaling, y=pdiffSettling] {\MSdata};
            \addlegendentry{Settling Time}

    \end{groupplot}

    \node [text width=1em,anchor=center] at ($(plots c1r1.north)-(0.5em,0.0em)$) {\subfloat[]{\label{fig:alg:dgp:sys2:eps}}};
    \node [text width=1em,anchor=center] at ($(plots c1r2.north)-(0.5em,0.0em)$) {\subfloat[]{\label{fig:alg:dgp:sys2:iter}}};
    \node [text width=1em,anchor=center] at ($(plots c1r3.north)-(0.5em,0.0em)$) {\subfloat[]{\label{fig:alg:dgp:sys2:pdiff1}}};
    \node [text width=1em,anchor=center] at ($(plots c1r4.north)-(0.5em,0.0em)$) {\subfloat[]{\label{fig:alg:dgp:sys2:pdiff2}}};
    \end{tikzpicture}
    
    \caption{Effect of scaling the continuous-time weight matrix $Q_{c}$ by $\alpha_1$ and holding $R_{c}$ constant on the Dual Gradient Projection method when solving System~\ref{sys:msd}.}
    \label{fig:alg:dgp:sys2}

\end{figure}
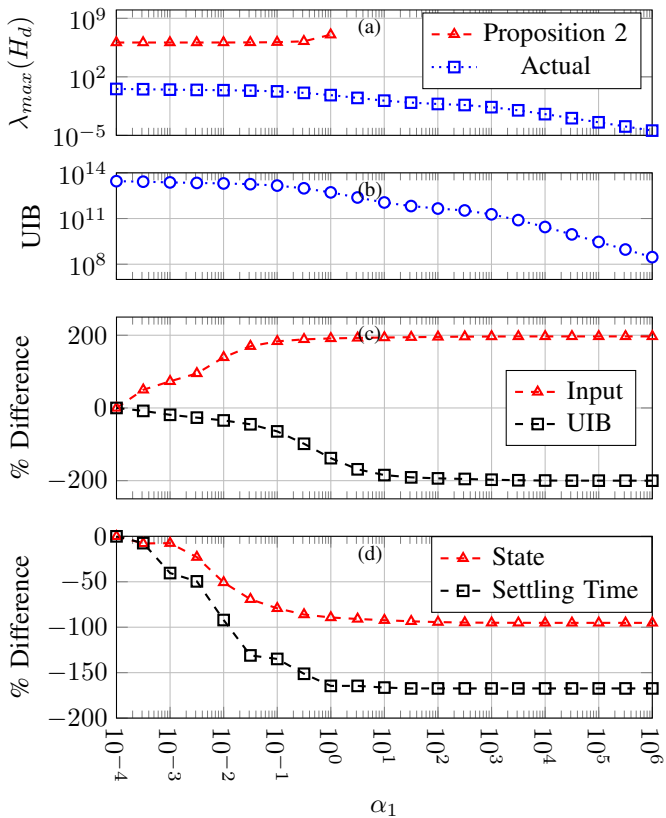

The Upper Iteration Bound for DGP given in \eqref{eq:dgp:uib} is dependent on the largest eigenvalue of both the primal and dual Hessian matrices, as well as the Upper Dual Bound.
 All three of these quantities are dependent upon the horizon length chosen; but while a horizon-independent bound on the UDB is not known, the eigenvalues can be bounded using the results in Section~\ref{sec:specRes}.
 Examining \eqref{eq:dgp:uib}, $\lambda_{max}$ of the primal Hessian only affects the UIB when a suboptimal solution to the dual problem is requested (e.g.\ $\epsilon_{z} \neq 0$), while $\lambda_{max}$ of the dual Hessian has a linear effect and the UDB has a quadratic effect on the UIB.

The numerical examples presented in Figures~\ref{fig:alg:dgp:sys1} and~\ref{fig:alg:dgp:sys2} use $N=20$ with $\epsilon_{z}=\epsilon_{\xi} = 0$ (e.g.\ solve the dual exactly and with no computation error) while $\epsilon_{g} = 10^{-4}$ and $\epsilon_{V} = 10^{-2}$.
 The UDB for System~\ref{sys:jonesMorari} shown in Figure~\ref{fig:alg:dgp:sys1:udb} is calculated by solving the MILP given in ~\cite{Necoara2014} using CPLEX with indicator constraints to implement the binary variables.
 The UDB for System~\ref{sys:msd} was upper bounded as $D{=}20000$ for all scaling factors by solving the same MILP, but terminating computation early and using the best objective value as the upper bound.
 As shown in Figure~\ref{fig:alg:dgp:sys1:iter}, the tight bounds from Theorem~\ref{thm:dual:toep:eigBounds} decrease the UIB by an order of magnitude compared with the estimate of $\lambda_{max}$ given in \cite{Patrinos2013_DGP}.

The UIB for DGP is affected by both the UDB and the maximal eigenvalues as the weight matrices are scaled.
 An interesting feature that appears when the exact UDB is used in System~\ref{sys:jonesMorari} is the interplay between $D$ and $\lambda_{max}(H_{d})$; specifically the slight peak in the UIB around $\alpha_{1} = 100$ before it drops off again.
 The UDB though appears to have an asymptotic structure at the two extremes for $\alpha_{1}$, while $\lambda_{max}(H_{d})$ is decreasing as $\alpha_{1}$ increases.
 A tradeoff in the closed-loop performance and computational complexity is also evident for DGP, with System~\ref{sys:jonesMorari} showing a 200\% increase in the UIB for a 200\% decrease in the input norm and 5\% increase in the state norm.

\section{Preconditioning}
\label{sec:preconditioning}

The spectral results presented in Section~\ref{sec:specRes:condHess} can be readily extended to analyze the case of a preconditioned Hessian matrix, and also to help design new preconditioners.

\subsection{Analysis of the Preconditioned Hessian}

For simplicity of description, we focus on the case when $H_{c}$ is symmetrically preconditioned as $L_{n}^{-1} H_{c} \trans{(L_{n}^{-1})}$ with a block-diagonal preconditioner $L_{n}$, thus guaranteeing that the preconditioned matrix is Toeplitz.
 This case is fairly standard in the MPC literature for first-order methods, since it guarantees that the structure of the feasible set is preserved over the preconditioning operation and that the preconditioned Hessian matrix is symmetric \cite{Richter2012}.
 Results can be derived for non block-diagonal preconditioners using \cite[Theorem~4.3]{Miranda2000} with $M^{-1}H_{c}$ where $M \coloneqq L_{n} \trans{L_{n}}$, but the handling of the cross-term matrix $S$ may not be as straightforward.
 

 
Since the preconditioner matrix $L_{n}$ is block-diagonal, its matrix symbol is simply $L$.
 The results in Section~\ref{sec:specRes:condHess} can then be extended to the preconditioned matrix by simply replacing $P_{H_{cQ}}$ in Theorems~\ref{thm:dense:PQ:hessEig} and~\ref{thm:dense:PLyap:hessEig} with $P_{H_{L}}$ given by
 \begin{equation}
     \label{eq:precond:toepSymbol}
     P_{H_{L}} \coloneqq \bar{L} P_{H_{cQ}} \trans{\bar{L}},
 \end{equation}
 where $\bar{L} \coloneqq L^{-1}$.
 The results in Theorem~\ref{thm:dense:Sterm:hessEig} can likewise be extended to the preconditioned case by redefining $P_{H_{n}}$ and $U$ as
 $ P_{H_{n}} \coloneqq \bar{L} ( P_{H_{cQ}} + P_{\bar{H}_{S}} ) \trans{\bar{L}} $
 and
 $
     U \coloneqq
     \begin{bmatrix}
     \trans{\bar{L}} \bar{L} \trans{B} S & \trans{\bar{L}} \bar{L} \\
     \trans{S} \bar{W}_{c} S & \trans{S} B \trans{\bar{L}} \bar{L}
     \end{bmatrix}
 $
 respectively, with $\bar{W}_{c}$ the controllability Gramian of the system with the input matrix $B \trans{\bar{L}}$.

\subsection{Preconditioner Design}

The Toeplitz structure of the Hessian matrix can also be exploited to design preconditioners for the primal problem.
 There is a rich literature of preconditioners for Toeplitz and circulant matrices, with a focus on designing the preconditioners independent of the size of the matrix (see \cite{Chan2007} and references therein).

For example, \cite{Chan1988} proposes an optimal circulant preconditioner for Toeplitz matrices that can be designed using only closed-form expressions.
 This can be used in designing a diagonal preconditioner for $H_{c}$, which is given in Corollary~\ref{corr:precond:circ}.
\begin{corollary}
    \label{corr:precond:circ}
    Let $H_{c}$ be the condensed primal Hessian matrix from Section~\ref{sec:specRes:condHess} and $P$ the solution to the discrete-time Lyapunov equation $A^{T} P A + Q = P$.
    The matrix $H_{c}$ can be symmetrically preconditioned as $L^{-1} H_{c} \trans{(L^{-1})}$, where $L$ is the lower-triangular Cholesky decomposition of $M$ and 
    \begin{equation*}
        M \coloneqq \trans{B} P B + \trans{S} B + \trans{B} S + R.
    \end{equation*}
\end{corollary}

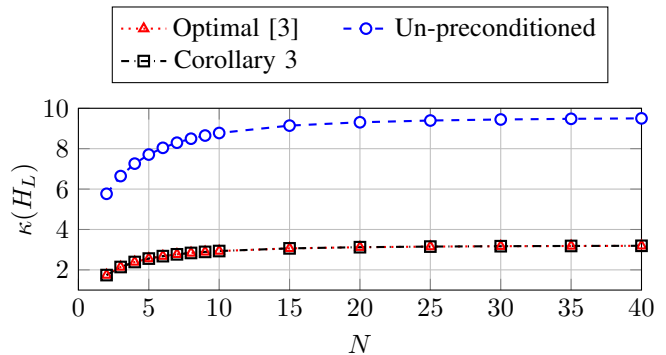
\begin{figure}[tb]
    \centering
    \pgfplotsset{actual/.style={dashed, thick, blue, every mark/.append style={solid, fill=white}, mark = *}}
    \pgfplotsset{richter/.style={dotted, thick, red, every mark/.append style={solid, fill=white}, mark = triangle}}
    \pgfplotsset{chan/.style={dashdotted, thick, black, every mark/.append style={solid, fill=white}, mark = square}}

    \begin{tikzpicture}
        \pgfplotstableread[col sep=comma]{figures/data/preconditioning_spectrum_JonesMorari.csv}{\preconddata}
        
        \begin{axis}[xmin   = 0,
                     xmax   = 40,
                     ymin   = 1,
                     ymax   = 10,
                     grid = major,
                     xlabel = {$N$},
                     ylabel = {$\kappa(H_L)$},
                     ylabel shift = -5 pt,
                     height = 4cm,
                     width  = 0.5\textwidth,
                     legend style={cells={align=left},/tikz/every even column/.append style={column sep=0.5cm}},
                     legend columns=2,
                     legend to name=precondSpectrumPlotleg,
                     title = {\pgfplotslegendfromname{precondSpectrumPlotleg}}]
                
                \addplot[richter] table [x=Horizon, y=Precond_O_K] {\preconddata};
                \addlegendentry{Optimal \cite{Richter2012}}
                
                \addplot[actual] table [x=Horizon, y=ActualK] {\preconddata};
                \addlegendentry{Un-preconditioned}
                
                \addplot[chan] table [x=Horizon, y=Precond_C_K] {\preconddata};
                \addlegendentry{Corollary~\ref{corr:precond:circ}}        
        \end{axis}
    \end{tikzpicture}
    
    \caption{The effect of preconditioning on the condition number of the condensed primal Hessian matrix for System~\ref{sys:jonesMorari}.}
    \label{fig:precond:spec}
\end{figure}

The block-diagonal preconditioner proposed in Corollary~\ref{corr:precond:circ} is independent of the horizon length, and is computable for any Schur stable system.
 The performance is also similar to that of the optimal preconditioner given in \cite{Richter2012}, as shown in Figure~\ref{fig:precond:spec}.
 Note that the optimal preconditioner must be recalculated at each horizon but the preconditioner in Corollary~\ref{corr:precond:circ} does not need to be.
 While the condition number of $H_{L}$ is the same for both preconditioners, the actual eigenvalue distribution is different.
 Corollary~\ref{corr:precond:circ} produces a lower minimum and maximum eigenvalue than the optimal preconditioner, which holds the lower eigenvalue constant at~1.
 This effect is most noticeable when the $Q$ matrix dominates the Hessian, as shown in Figures~\ref{fig:precond:scale:specMin} and~\ref{fig:precond:scale:specMax}.
 
Overall, the spectrum of the preconditioned matrix $H_{L}$ has the same behavior as the spectrum of $H_{c}$ when the weighting matrices are scaled.
 The main difference being that the condition number in the $Q$ dominating region is smaller when a preconditioner is used, as shown in Figure~\ref{fig:precond:scale:cond}.
 An interesting thing to note is that the optimal preconditioner from \cite{Richter2012} actually becomes incalculable for large ratios of $Q$ to $R$ in System~\ref{sys:jonesMorari}.
 The example suggests that for $\beta$ above 400, the optimization problem in \cite{Richter2012} becomes infeasible but that the proposed preconditioner in Corollary~\ref{corr:precond:circ} is still calculable.

\begin{figure}[tb]
    \centering
    \pgfplotsset{actual/.style={dashed, thick, blue, every mark/.append style={solid, fill=white}, mark = *}}
    \pgfplotsset{richter/.style={dotted, thick, red, every mark/.append style={solid, fill=white}, mark = triangle}}
    \pgfplotsset{chan/.style={dashdotted, thick, black, every mark/.append style={solid, fill=white}, mark = square}}

    \begin{tikzpicture}
        \pgfplotstableread[col sep=comma]{figures/data/preconditioning_WeightScaling_JonesMorari.csv}{\preconddata}
        
        \begin{groupplot}[group style = {group name=plots, group size = 1 by 3, horizontal sep=0cm, vertical sep=0.5cm}]
            
            \nextgroupplot[xmin   = 0.0001,
                           xmax   = 10000,
                           ymin   = 0.1,
                           ymax   = 10000,
                           xmajorticks = false,
                           xtick distance = 10,
                           grid = major,
                           ylabel = {$\lambda_{min}(H_L)$},
                           height = 4cm,
                           width  = 0.48\textwidth,
                           ymode  = log,
                           xmode  = log,
                           title = {\pgfplotslegendfromname{precondScalingPlotleg}}]
                 
                \addplot[actual] table [x=Scaling, y=minE_actual] {\preconddata};
                
                \addplot[chan] table [x=Scaling, y=minE_C_precond] {\preconddata};
                
                \addplot[richter] table [x=Scaling, y=minE_O_precond] {\preconddata};

            \nextgroupplot[xmin   = 0.0001,
                           xmax   = 10000,
                           ymin   = 0.1,
                           ymax   = 1000000,
                           xmajorticks = false,
                           xtick distance = 10,
                           grid = major,
                           ylabel = {$\lambda_{max}(H_L)$},
                           ylabel shift = -7 pt,
                           height = 4cm,
                           width  = 0.48\textwidth,
                           ymode  = log,
                           xmode  = log]

                \addplot[actual] table [x=Scaling, y=maxE_actual] {\preconddata};
                
                \addplot[chan] table [x=Scaling, y=maxE_C_precond] {\preconddata};
                
                \addplot[richter] table [x=Scaling, y=maxE_O_precond] {\preconddata};

            \nextgroupplot[xmin   = 0.0001,
                           xmax   = 10000,
                           ymin   = 1,
                           ymax   = 1000,
                           grid = major,
                           xlabel = {$\beta$},
                           ylabel = {$\kappa(H_L)$},
                           x tick label style={rotate=90,anchor=east},
                           xtick distance = 10,
                           height = 4cm,
                           width  = 0.48\textwidth,
                           ymode  = log,
                           xmode  = log,
                           legend style={cells={align=left},/tikz/every even column/.append style={column sep=0.5cm}},
                           legend columns=2,
                           legend to name=precondScalingPlotleg]
                
                \addplot[richter] table [x=Scaling, y=Precond_O_K] {\preconddata};
                \addlegendentry{Optimal \cite{Richter2012}}
                
                \addplot[actual] table [x=Scaling, y=ActualK] {\preconddata};
                \addlegendentry{Un-preconditioned}
                
                \addplot[chan] table [x=Scaling, y=Precond_C_K] {\preconddata};
                \addlegendentry{Corollary~\ref{corr:precond:circ}}

        \end{groupplot}
        \node [text width=1em,anchor=center] at ($(plots c1r1.north)-(1.1em,0.7em)$) {\subfloat[]{\label{fig:precond:scale:specMin}}};
        \node [text width=1em,anchor=center] at ($(plots c1r2.north)-(1.1em,0.7em)$) {\subfloat[]{\label{fig:precond:scale:specMax}}};
        \node [text width=1em,anchor=center] at ($(plots c1r3.north)-(1.1em,0.7em)$) {\subfloat[]{\label{fig:precond:scale:cond}}};
    \end{tikzpicture}
    
    \caption{The effect of two different preconditoners on the extremal eigenvalues and the condition number of the condensed primal Hessian matrix $H_{c}$ when the $Q$ and $R$ matrices are scaled by $\beta$ in System~\ref{sys:jonesMorari}.}
    \label{fig:precond:scale}
\end{figure}
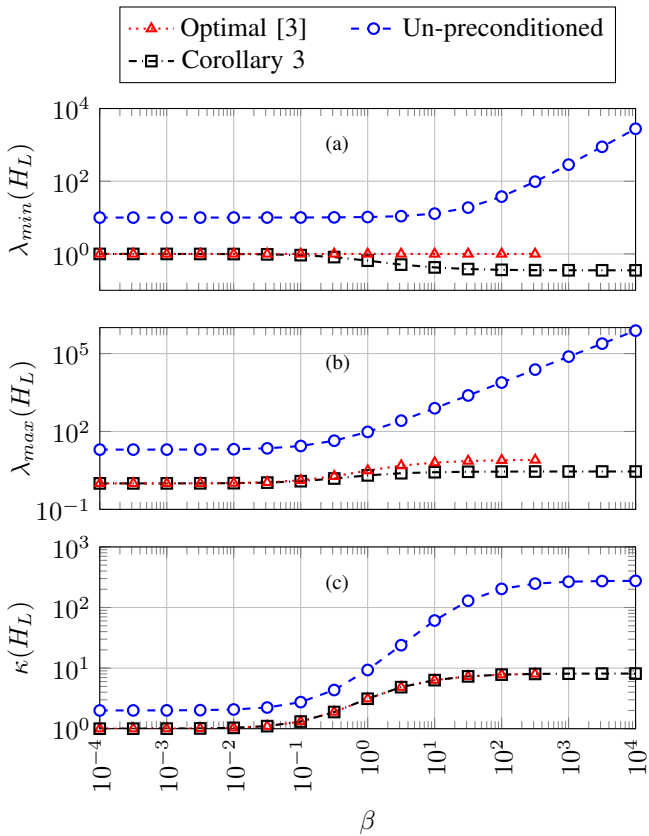

\section{Conclusions}
\label{sec:conclusion}

In this paper, we have examined how the computational complexity bounds for the Fast Gradient Method and Dual Gradient Projection method are influenced by the desired system performance (e.g.\ weighting matrix selection).
 The complexity bounds for FGM and DGP demonstrate distinct regions where the bound is influenced by the spectrum of the individual weighting matrices.
 Additionally, the complexity bounds of FGM and DGP behave differently under cost function scaling; the FGM bound increased as $Q$ dominated while the DGP bound decreased as $Q$ dominated.
 This suggests that not pre-determining the algorithm and instead having it as a design variable could benefit the overall system design.

To derive the computational complexity bounds, we derived a system-theoretic method for analyzing the primal and dual Hessian matrices by viewing the matrices as Toeplitz operators.
 This method allows for horizon-independent bounds of the extremal eigenvalues and condition number to be computed using tools such as the $H_{\infty}$ norm; removing the need to form large matrices to experimentally estimate the values.

While we applied these bounds to computing computational complexity, they can also be applied to the design of the actual computing hardware for the algorithms.
 Fixed-point implementations of algorithms such as FGM \cite{Richter2012}, DGP \cite{Patrinos2013_DGP}, and Proximal Newton \cite{Guiggiani2014_proximalNewton} utilize the condition number and extremal eigenvalues to bound the round-off error that is present in the computations.
 Horizon-independent spectral bounds can then be used to guarantee hardware designs are compatible with any horizon length desired, removing the need to re-synthesize the designs if the prediction horizon were to be changed and allowing run-time variation of the horizon length (e.g.\ for variable-horizon controllers) in fixed-point implementations.

This system-theoretic approach can also be used to examine preconditoned Hessian matrices.
 We derived a preconditioner that is equivalent to the optimal preconditioner in \cite{Richter2012}, but is computable in closed-form using small matrices.
 Our framework allows for preconditioning to be viewed from the system-theoretic perspective, with the preconditioner being an input compensator for the predicted system.
 This viewpoint suggests new design techniques, such as methods from $H_{\infty}$ loop-shaping, to compute preconditioner matrices may exist.
 
In the future, these bounding results can be incorporated into controller-design methods to estimate the computational resources needed for a system and be used inside design optimization.
 We have demonstrated that this trade-off between computational resources and control performance is a worthwhile area to explore; with an example system showing that a 5\% reduction in the state 2-norm requires a 188\% increase in the computational complexity of the control algorithm.

\appendices

\renewcommand{\theequation}{\thesection.\arabic{equation}}
\renewcommand{\theproposition}{\thesection.\arabic{proposition}}

\section{Matrices for the CLQR Problem}
\label{app:mpcMatrices}
 \begin{equation*}
     \Phi \coloneqq \begin{bmatrix}
      A \\ A^2 \\ A^3 \\ \vdots \\ A^N
     \end{bmatrix},
     \Gamma \coloneqq \begin{bmatrix}
      B & 0 & 0 & & 0\\
      AB & B & 0 & & 0\\
      A^2B & AB & B & & 0 \\
      \vdots & & & \ddots & \vdots\\
      A^{N-1}B & A^{N-2}B & A^{N-3}B & \cdots & B
     \end{bmatrix},
 \end{equation*}
 \begin{equation*}
    \bar{R} \coloneqq I_{N} \otimes R,\quad
    \bar{S} \coloneqq \begin{bmatrix}
    I_{N-1} \otimes S & 0\\
    0 & 0
    \end{bmatrix},\quad
    \bar{Q} \coloneqq \begin{bmatrix}
     I_{N-1} \otimes Q & 0\\
     0 & P
     \end{bmatrix},
 \end{equation*}
 \begin{equation*}
    G \coloneqq \bar{D} \Gamma + \bar{E}, \quad
    \bar{D} \coloneqq I_{N} \otimes 
     \begin{bmatrix}
      D\\
      0_{l \times n} \\
     \end{bmatrix},\quad
    \bar{E} \coloneqq I_{N} \otimes
     \begin{bmatrix}
      0_{j \times m}\\
      E
     \end{bmatrix},
 \end{equation*}
 \begin{equation*}
    F \coloneqq -\bar{D}
     \Phi, \quad
    g \coloneqq \mathbf{1}_{N} \otimes
     \begin{bmatrix}
     c_x \\
     c_u
     \end{bmatrix}.
 \end{equation*}

\section{Proofs}

\subsection{Proof of Lemma~\ref{lem:scaling:prelim:limitingValues}}
\label{app:proof:limitingValues}
\begin{proof}
    Since $S$ is arbitrary, $H_{c} = H_{cQ} + H_{n} - H_{e}$.
    We begin by computing $a_l \coloneqq \lim_{n \to \infty} \frac{1}{n} \trace{H_{c}}$.
    $H_{e}$ has finite rank (from Lemma~\ref{lem:dense:Sterm:HeRank}), so its limit will go to $0$ leaving only the Topelitz component
    \begin{equation}
        a_l = \lim_{n \to \infty} \frac{1}{n} \trace{H_{cQ} + H_{n}}. \label{eq:traceLimits:alLimitTrace}
    \end{equation}
    Using Szeg{\"o}'s limit theorem for block Toeplitz matrices \cite[Theorem 6.5]{Gutierrez-Gutierrez2012_blockSurvey}, the limit as $n \to \infty$ of a function $f$ of the eigenvalues of a Toeplitz matrix can be transformed into a definite integral of $f$ applied to the eigenvalues of the matrix symbol.
    Since the trace is the sum of the eigenvalues, \eqref{eq:traceLimits:alLimitTrace} becomes
    \begin{align*}
        a_l = \frac{1}{2m\pi}\int_{0}^{2\pi} \trace{ P_{H_{cQ}}(e^{j \omega}) + P_{H_{n}}(e^{j \omega})} d\omega.
    \end{align*}
    Separating over the addition and expanding,
    \begin{multline*}
        a_l = \frac{1}{2m\pi}\biggl( \int_{0}^{2\pi} \trace{\ctrans{P_{\Gamma}(e^{j\omega})} Q P_{\Gamma}(e^{j\omega})} + \trace{R}\\
        + 2\trace{ \trans{S} P_{\Gamma}(e^{j\omega}) d\omega }\biggr).
    \end{multline*}
    Using the definition of the $H_{2}$ norm and Frobenius norm, $a_{l}$ becomes
    \begin{equation*}
        a_l = \frac{1}{m}\left( \Htwonorm{G_Q}^2 + 2I_{1} + \Fronorm{\sqrtm{R}}^2 \right).
    \end{equation*}
        
    Computing $b_l$ can be done in a similar manner to $a_l$.
    Note that $ \text{Rank}(AB) \leq \min\{ \text{Rank}(A), \text{Rank}(B)\}$, meaning anything multiplied by $H_{e}$ will have finite rank.
    This means those terms will go to $0$ in the limit, leaving only the Topelitz component
    \begin{equation}
    b_l = \lim_{n \to \infty} \frac{1}{n} \trace{(H_{cQ} + H_{n})^2}. \label{eq:traceLimits:blLimitTrace}
    \end{equation}
    Applying the Szeg{\"o} limit theorem to \eqref{eq:traceLimits:blLimitTrace} and expanding gives
    \begin{equation*}
        b_l = \frac{1}{2m\pi}\int_{0}^{2\pi} \trace{(P_{H_{cQ}}(e^{j \omega}) + P_{H_{n}}(e^{j\omega}))^2 } d\omega \\
    \end{equation*}
    Expanding the integrand, and simplifying using properties of the trace and the definition of the $H_{2}$ and Frobenius norms produces the final result.
\end{proof}

\subsection{Proof of Lemma~\ref{lem:condScaling:condLowerBound}}
\label{app:condAsympBoundProof}
\begin{proof}
    We begin with the lower bound for the condition number of a matrix of dimension $n$ presented in \cite[Corollary 2.3]{Styan1983}:
    \begin{equation}
        \label{eq:asympEigBoundProof:initialLower}
        \kappa(H_{c}) \geq 1 + \frac{2s}{a - \sfrac{s}{p}}
    \end{equation}
    with $p \coloneqq \sqrt{n - 1}$, $s \coloneqq \sqrt{b - a^2}$, and $a$ and $b$ from Lemma~\ref{lem:scaling:prelim:limitingValues}.
    
    To determine the asymptotic bound, we take the limit of \eqref{eq:asympEigBoundProof:initialLower} to find
    \begin{equation*}
        \label{eq:asympEigBoundProof:mainEq}
        \lim_{n \to \infty} 1 + \frac{2s}{a - \sfrac{s}{p}} = 1 + 2\frac{\sqrt{b_{l} - a_{l}^2}}{a_{l}}.
    \end{equation*}
\end{proof}

\subsection{Proof of Theorem~\ref{thm:condScaling:scaledBound}}
\label{app:scaledBoundProof}
\begin{proof}
    Starting with the lower bound from Lemma~\ref{lem:condScaling:condLowerBound}, we define
    $v \coloneqq \sqrt{b_{l} - a_{l}^2}$ and $w \coloneqq a_{l}$.
    Substituting in the norm scalings from Lemma~\ref{lem:scaling:prelim:normScaling} into $w$ gives
    \begin{equation*}
        w = \alpha_1 \Htwonorm{G_{Q}}^2 + \alpha_2 \Fronorm{\sqrtm{R}}^2,
    \end{equation*}
    which is the final version of the denominator.
    
    Substituting in the norm scalings from Lemma~\ref{lem:scaling:prelim:normScaling} into $v$ and expanding the square gives
    \begin{multline*}
        v = m \left( \alpha_1^2 I_{6} + 2\alpha_1 \alpha_2\Htwonorm{G_{QR}}^2 + \alpha_2^2\Fronorm{R}^2 \right)\\
         - \alpha_1^2 \Htwonorm{G_{Q}}^4 - \alpha_1 \alpha_2 \Htwonorm{G_{Q}}^2 \Fronorm{\sqrtm{R}}^2 - \alpha_2^2 \Fronorm{\sqrtm{R}}^4
    \end{multline*}
    Grouping the terms by the $\alpha_1$ and $\alpha_2$ coefficients leads to
    \begin{multline*}
        v = \alpha_1^2 \left( m I_6 - \Htwonorm{G_{Q}}^4 \right)
          + \alpha_2^2 \left( m\Fronorm{R}^2 - \Fronorm{\sqrtm{R}}^4 \right) \\
          + 2 \alpha_1 \alpha_2 \left( m\Htwonorm{G_{QR}}^2 - \Htwonorm{G_{Q}}^2 \Fronorm{\sqrtm{R}}^2 \right)
    \end{multline*}
    which gives the final version of the numerator.
\end{proof}

\balance
\bibliography{references/tuning,references/examples,references/references,references/toeplitz,references/intro}

\end{document}